\documentclass[ejs]{imsart}

\RequirePackage[OT1]{fontenc}
\RequirePackage{amsthm,amsmath, amssymb}
\RequirePackage[sectionbib]{natbib}
\RequirePackage[colorlinks,citecolor=blue,urlcolor=blue]{hyperref}
\usepackage{graphicx}

\pubyear{2021}
\volume{0}
\issue{0}
\firstpage{1}
\lastpage{8}

\startlocaldefs
\numberwithin{equation}{section}
\theoremstyle{plain}
\newtheorem{thm}{Theorem}[section]
\endlocaldefs

\newtheorem{lemma}{Lemma}[section]
\newtheorem{corollary}{Corollary}[section]
\newtheorem{proposition}{Proposition}[section]

\newcommand{\beq}{\begin{eqnarray*}}
\newcommand{\eeq}{\end{eqnarray*}}
\newcommand{\beqn}{\begin{eqnarray}}
\newcommand{\eeqn}{\end{eqnarray}}

\newcommand{\ra}{\rightarrow}

\newcommand{\bi}{\begin{itemize}}
\newcommand{\ei}{\end{itemize}}

\newcommand{\be}{\begin{equation}}
\newcommand{\ee}{\end{equation}}

\newcommand{\pkg}[1]{{\fontseries{b}\selectfont #1}}

\newcommand{\bfm}[1]{\ensuremath{\mathbf{#1}}}

          \def\cG{{\cal  G}}
          
\def\bi{\bfm i}

          \def\cN{{\cal  N}}

     \def\bS{\bfm S}

\def\eps      {\varepsilon}
\newcommand{\lbl}{\label}

\newcommand{\ignore}[1]{}{}

\newcommand{\Loss}{{\mathcal L}}

\newcommand{\e}{\mathbb{E}}
\newcommand{\p}{\mathbb{P}}

\newcommand{\mS}{\mathcal{S}}

\newcommand{\mC}{\mathcal{C}}

\newcommand{\mR}{\mathbf{R}}

\newcommand{\htau}{\hat{\tau}}

\newcommand{\argmin}{\mathop{\rm arg\min}}

\numberwithin{equation}{section}
\theoremstyle{definition}

\theoremstyle{definition}
\newtheorem{condition}{Condition}[section]

\def\be{\begin{equation}}
\def\ee{\end{equation}}

\newcommand{\bve}{\mbox{\boldmath$\varepsilon$}}

\newcommand{\btheta} {\boldsymbol{\theta}}

\newcommand*\dt{\mathop{}\!\mathrm{d}}

\newcommand{\sigst}{\mathcal{S}} 

\makeatletter
\let\oldabs\abs
\def\abs{\@ifstar{\oldabs}{\oldabs*}}
\let\oldnorm\norm
\def\norm{\@ifstar{\oldnorm}{\oldnorm*}}
\makeatother

\begin{document}

\begin{frontmatter}
\title{Consistency of a range of penalised cost approaches for detecting multiple changepoints}
\runtitle{Penalised Cost Approaches for Changepoints}

\begin{aug}
\author{\fnms{Chao} \snm{Zheng}$^{\dagger,*}$ {}}
\author{\fnms{Idris} \snm{Eckley}$^{\dagger}${}}
\and
\author{\fnms{Paul} \snm{Fearnhead}$^{\dagger}${}}

\address{$^{\dagger}$ Department of Mathematics and Statistics\\
Lancaster University
}

\address{$^{*}$ School of Mathematical Sciences\\
University of Southamtpon
}

\runauthor{C. Zheng et al.}

\affiliation{Some University and Another University}

\end{aug}

\begin{abstract}
A common approach to detect multiple changepoints is to minimise a
measure of data fit plus a penalty that is linear in the number of
changepoints. This paper shows that the general finite sample behaviour
of such a method can be related to its behaviour when analysing data
with either none or one changepoint. This property results in simpler conditions
for verifying whether the method will consistently estimate the number
and locations of the changepoints. We apply and demonstrate the
usefulness of these simple conditions for a range of changepoint problems. Our new
results include a weaker requirement on the choice of penalty to
have consistency in a change-in-slope model; and the first results for
the accuracy of recently-proposed methods for detecting spikes.
\end{abstract}


\end{frontmatter}

\section{Introduction}   \lbl{sec:intro}
Detecting changepoints is a long-standing problem in statistics, dating
at least as far back as \cite{page1955test}. In recent years, there
has been an explosion of research into methods for
detecting multiple changes in data, stimulated by the growing 
need  across many and diverse application areas. For
example, changepoint detection methods have been applied in finance \citep{preup2015detection}, bioinformatics \citep{cribben2017estimating}, network traffic \citep{lung2012distributed} and climatology \citep{itoh2010change}.

There are a number of generic ways to detect multiple changepoints, many
of which are based on recursively or repeatedly applying a method that
detects and locates single changepoint. These include binary
segmentation \citep{vostrikova1981detecting} and its variants such as wild binary segmentation
\citep{fryzlewicz2014wild} and circular binary segmentation \citep{olshen2004circular};
simultaneous multiscale changepoint
estimation \citep{frick2014multiscale,li2016FDR}; scan
statistics \citep{eichinger2018MOSUM}; and the narrowest-over-threshold approach
\citep{baranowski2016narrow}. This paper considers a different class of
popular changepoint algorithms, that aims to jointly detect multiple
changepoints through minimising a penalised cost. This cost involves a
measure of fit to the data together with a penalty that increases
proportionally with the number of changepoints. These penalised cost ideas are
closely related to penalised likelihood methods and model choice
approaches \citep{birge2001gaussian} for linear models that employ an $L_0$ penalty. 

In practice, different applications often require
the ability to detect different types of change, or have different types
of data structure. The simplest case is detecting multiple changes-in-mean in Gaussian data, where the goal is equivalent to find the optimal piecewise-constant mean function. Fast algorithms for penalised cost approaches with an expected run time that can be linear in the amount of data \citep{killick2012optimal, maidstone2018optimal} have been developed. These algorithms have been widely used in applications from  RNA sequencing \citep{cleynen201comparing} to analysis of historical warfare \citep{fagan2019changepoint}.  For detecting other types of change, for example, changes in slope \citep{fearnhead2018detecting}, spikes plus exponential decay\citep{jewell2018exact}, and changes in mean of heavy-tailed data  \citep{fearnhead2019changepoint}, efficient algorithms have also been proposed, and have
been implemented to solve real-world problems such as 
determining the exact moment in time at which a neuron spikes in calcium image data\citep{jewell2019fast}.

In terms of theoretical analysis, existing theory is well-developed for penalised cost approaches for detecting change-in-mean \citep{yao1978estimating, lavielle2000least, tickle2018parallelisation}.
For example, a specific
penalty value is known to give consistent estimates of the number of
changepoints and empirical results suggest that this penalty
choice is tight: with lower values frequently leading to over-estimating
the number of changes. However, separate results needs to be carried out if we want to look at different types of change. Whilst detecting different types of change seems to be similar statistical problems to detecting changes in mean, it is fundamentally more challenging in certain cases due to the complexity of underlying signals.  Asymptotic results about consistent estimates of the number and locations of changepoints in those cases are much more limited and weaker, if not unavailable, in existing literatures, e.g., often opposing an unspecific value of penalty that depends on a loosely defined large enough constant  \citep{fearnhead2018detecting} and/or requiring a finite upper bound on the number of changes as in \cite{yao1978estimating}.  
To our knowledge, the most general results are in \cite{boysen2009consistencies} which provides a simple argument that demonstrates
consistency for a wide class of changepoint models if we measure fit via
a residual sum of squares, and have a penalty that increases faster than the
logarithm of the number of data points. In this paper we develop a general statistical theory for penalised cost approaches of detecting different types of changes, and, in particular, provide theoretical guidance on on how the penalty for each additional changepoint, should be chosen.


 As the first contribution of this paper, we establish a general framework of penalised cost
approaches, which applies to all multiple changepoint models that admit an additive cost as the measure of fit, and show that the performance of the approaches is related to their behaviour when analysing data with either no
change or one changepoint. We call these two specific cases as local regions. Analysing the behaviour of penalsed cost methods in local regions is much simpler than for the case of multiple changes.
Informally our results show, subject to additional conditions, that a
choice of penalty that leads to consistent estimates of the number of
changes when there is one or no changepoints will directly lead to a consistent
estimator when there are multiple changes. Therefore, for different types of changepoint problem, we only need to focus on their properties in local regions. We propose a set of general conditions of detectability of changes across all local regions, that will be sufficient to provide consistent estimator of the number and locations of multiple changes for the entire data.

As the second contribution, 
 we apply the proposed framework to different multiple 
changepoint problems as applications, showing that the local region conditions are satisfied, and hence obtain the theoretical results on the choice of penalty and the corresponding consistent estimates of number of changepoints and their locations. These lead to new results for detecting
changes-in-slope, with a weaker condition on the penalty that is needed
to obtain consistent estimates comparing to existing results in \cite{fearnhead2018detecting}; and lead to the first theoretical
results for the problem of detecting spikes in an exponentially decaying
signal. Using our framework, similar consistency results can be easily derived for many other multiple changepoint detection problems.

\section{General Framework} \lbl{sec:method}


\subsection{Problem setup} \lbl{sec:2.1}

Consider a general changepoint model with $T$ observations $x_{1:T}=\{x_1,\dots,x_T\}$  ordered in sequence, for example by time or position along a chromosome. 
Assume that there are a set of $m$ changepoints, at ordered locations, $\tau_{1:m}=\{\tau_1,\dots,\tau_m\}$, with $0<\tau_1<\tau_2<\cdots<\tau_m<T$. This will partition the data into $m+1$ distinct segments with the $j$-th segment including the observations $x_{\tau_{j-1}+1:\tau_j}=\{x_{\tau_{j-1}+1},\dots,x_{\tau_{j}}\}$, where we write $\tau_0=0$ and $\tau_{m+1}=T$. In other words,  there is a  common structure for the data within each single segment, but the nature of this structure can change between segments. We characterise this structure using segment specific parameters $\btheta_{1:m+1}=\{\btheta_1,\dots, \btheta_{m+1}\}$, which depending on the application, could, for example, be the mean  changes of the data between segments, or the variance changes, or both, among many other possibilities.

We wish to estimate both the number and locations of the changepoints. To this end we focus on the $L_0$ penalised cost methods. These introduce a segment cost, which measures the fit to data within the segment. Often appropriate costs are specified by modelling the data, and setting the cost to be the negative of the log-likelihood under such a model. For a segment with data $x_{s:e}$ for some $e> s$, we have a segment cost $\mC(x,s,e; \btheta)$ that will depend on the segment specific parameters $\btheta$. We will assume that the cost is additive over data points: 
\begin{equation} \label{eqn:add}
\mC(x,s,e; \btheta)=\sum_{t=s}^e c_t(x_{1:T},\btheta ).
\end{equation}
where, $c_t(x_{1:T},\btheta)$, the cost associated with data point $x_t$ can also depend on other data-points, which allows dependency between $x_{1:T}$. This additive assumption is a quite generally property for segment cost, which holds for most problems in changepoint literature.

We then define the cost for fitting a set of $m$ changepoints, $\tau_{1:m}$, with associated parameters $\btheta_{1:m+1}$ as 
\begin{align*}
\sum_{j=1}^{m+1}\mC(x,{\tau_{j-1}+1, \tau_j}; \btheta_j).
\end{align*}
We minimise over the parameters $\btheta_{1:m+1}$ simultaneously to define the cost associated with the segmentation:
\begin{align*}
\Loss\left(x_{1:T};\tau_{1:m}\right)=\min_{\btheta_{1:m+1}}\sum_{j=1}^{m+1}\mC(x,{\tau_{j-1}+1, \tau_j}; \btheta_j),
\end{align*}
where, potentially, the minimisation can be subject to constraints on the segment parameters. For simplicity, here we assume the constraints fix the relationship between the parameters of neighbouring segments.
 For example, we may enforce strict monotonicity in a change-in-mean model, i.e, $\theta_{j+1}\in(\theta_j,\infty)$, where $\theta_j$ is mean value in $j$-th segment; or in a change-in-slope model $\theta_j$ will specify a linear function for the mean signal within $j$-th segment, and the constraints would enforce continuity at the changepoints for two consecutive linear functions. Our results in this paper will remain valid for more general constraints as long as (\ref{lem.property.1}) and (\ref{lem.property.2}) in Lemma \ref{lem:1}, given below, still hold.

We can extend our definition of cost so that it applies to a subset of data $x_{s:e}$ and a series of changepoint locations $s\le \tau_{u:v}<e$. Using the notation that $\tau_{u-1}=s-1$ and $\tau_{v+1}=e$, we have
\begin{align*}
\Loss\left(x_{s:e};\tau_{u:v}\right)=\min_{\btheta_{u:v+1}}\sum_{j=u}^{v+1}\mC\left(x, {\tau_{j-1}+1, \tau_j}; \btheta_j\right).
\end{align*}
We take the convention that if the set of changepoints contains changes outside the region of data, then these changes are ignored.
For example, if $\tau<s$ or $\tau\ge e$ then  $\Loss(x_{s:e};\tau_{u:v},\tau)=\Loss(x_{s:e};\tau_{u:v})$. Let $\Loss(x_{s:e}; \varnothing)$ denote the segmentation cost where there is no changepoint between $s$ and $e$.

If we know the number of changepoints, $m$, it would be natural to estimate their locations by the set $\htau_{1:m}$, which minimises $\Loss(x_{1:T}; \tau_{1:m})$. However, in practice we need to also estimate the number of changepoints. Therefore we consider methods that estimate $m$ and $\tau_{1:m}$ simultaneously as the value that minimises the $L_0$ penalised cost:
\begin{equation} \lbl{def.est}
\argmin_{(m, \tau_{1:m})}\big\{\Loss\left(x_{1:T};\tau_{1:m}\right)+\beta m\big\},
\end{equation}
where $\beta>0$ is a user-defined tuning parameter that penalises the addition of each changepoint. We call (\ref{def.est}) minimisation of $L_0$ penalised cost as $m\beta$ can be viewed as an $L_0$ penalty on the difference in segment parameters associated with neighbouring time-points.


We will use the superscript $\ast$ to denote the true changepoint locations and parameter values. That is, the true model will have $m^\ast$ changepoints, at locations $\tau^\ast_{1:m^\ast}$, and with segment parameters $\btheta^\ast_{1:m^\ast+1}$. If we impose constraints when calculating the cost of a segmentation, we require the true segment parameters to satisfy those constraints. We will also define $\Loss^\ast(x_{s:e})$ to be the cost of fitting data $x_{s:e}$ with the true set of changepoints and the true parameters, that is
\begin{align*}
 \Loss^\ast(x_{s:e})= \mC(x, {s, \tau^\ast_u}; \btheta^\ast_{u})+\mC(x, {\tau^\ast_v, e}; \btheta^\ast_{v+1})+
 \sum_{j=u+1}^{v} \mC(x, {\tau^\ast_{j-1}, \tau^\ast_j}; \btheta^\ast_j),
\end{align*}
where $u$ and $v$ are defined so that $\tau^\ast_u$ and $\tau^\ast_v$ are, respectively, the first and last true change between $s$ and $e$; and if $u=v$ we can set the summation part to be 0. If $x_{s:e}$ does not contain a change then $\Loss^\ast(x_{s:e})=\mC(x, {s, e}; \btheta^\ast_u)$ where $u$ is the index of the true segment that $x_{s:e}$ lie in.

The follwing properties of the cost, which follow from the additive assumption (\ref{eqn:add}),  will be important for results in subsequent sections.
\begin{lemma} \label{lem:1}
 Assume  (\ref{eqn:add}) holds, then for any $s\le r<e$ and $s\le \tau_{u:v}<e$,  we have
\begin{equation} \label{lem.property.1}
\Loss(x_{s:e};\tau_{u:v})
\ge \Loss(x_{s:r};\tau_{u:v})+\Loss(x_{r+1:e};\tau_{u:v}), \mbox{ and}
\end{equation}
\begin{equation}
 \Loss^*(x_{s:e})=\Loss^*(x_{s:r})+\Loss^*(x_{r+1:e}).\label{lem.property.2}
\end{equation}
\end{lemma}

\subsection{Local region conditions} \lbl{sec:2.2}


Our aim is to build general conditions under which estimating the number and locations of the changepoints via a $L_0$ penalised cost approach will be consistent, and to quantify the accuracy within which the locations are estimated. We will achieve this by relating properties of penalised cost approach when analysing data with multiple changes in $x_{1:T}$ to its properties when analysing data with either zero or one changepoint in a local region, i.e, $x_{t+1: t+n}$ or $x_{t+1: t+2n}$, respectively, where $n$ or $2n$ is the number of data points the local region.

To this end we introduce the following conditions that govern the value of the penalised cost procedure when fitting data simulated with either zero or one true changepoint with either the correct or too many number of changepoints. These conditions need to apply for our assumed data generating mechanism and our choice of penalised cost. We assume the data generating mechanism is parameterised by the set of changepoints and the segment parameters.



\begin{condition}\label{cond.1}   There exists increasing positive constants $\gamma_n^{(1)}$ and $\gamma^{(2)}_n$, such that $\gamma_n^{(1)}$ and $\gamma^{(2)}_n\ra\infty$ as $n\ra \infty$; and there exists positive  numbers $a(\gamma,n)$ and $b(\gamma,n)$ increasing in $\gamma$,
 and positive functions, $p_j(\gamma,n)$ such that $p_j(\gamma,n)\ra 0$ as $\gamma \ra \infty$, for $j\in \{1,2,3,4\}$, such that:
\
\begin{itemize} 
 \item[(i)]  Let $\bS_{1,n}(t)=x_{t+1:t+n}$ be a segment between $t+1$ and $t+n$ that includes no changepoint.  If $\gamma\geq \gamma^{(1)}_n$,
$$
\max_{t}\, \p\left(\min_{k\ge 1 ,\tau_{1:k}}\{\Loss\left(\bS_{1,n}(t);\tau_{1:k}\right)+k\gamma\}-\Loss^\ast\left(\bS_{1,n}(t)\right)\le a(\gamma,n)\right)\le p_1(\gamma,n),
$$
$$
\max_{t}\, \p\bigg( \Loss^\ast\left(\bS_{1,n}(t)\right)- \Loss\left(\bS_{1,n}(t);\varnothing\right) \geq b(\gamma,n) \bigg) \le p_2(\gamma, n),
$$
where probability is with respect to the data generating mechanism for $\bS_{1,n}(t)$.
\item[(ii)] Let $\bS_{2,n}(t)=x_{t+1:t+2n}$ be a segment between $t+1$ and $t+2n$ that has a single changepoint which is at $t+n$. If  $\gamma\geq \gamma^{(2)}_n$, we have
$$
\max_{t}\, \p\left(\min_{k\ge 2,\tau_{1:k}}\{\Loss\left(\bS_{2,n}(t);\tau_{1:k}\right)+(k-1)\gamma\}-\Loss^\ast\left(\bS_{2,n}(t)\right)\le a(\gamma,2n)\right)\le p_3(\gamma,n),  
$$
$$
\max_{t}\, \p\left( \Loss^\ast\left(\bS_{2,n}(t)\right)- \min_{\tau_1} \Loss\left(\bS_{2,n}(t);\tau_1 \right) \geq b(\gamma, 2n) \right) \le p_4(\gamma,n),
$$
where probability is with respect to the data generating mechanism for $\bS_{2,n}(t)$.
\end{itemize}

\end{condition}

The above condition bounds the reduction in the penalised cost, if we have a penalty of $\gamma$ for adding a changepoint, that can be obtained by fitting too many changes. Note that this penalty of $\gamma$ , is specific for local regions $\bS_{1,n}(t)$ and $\bS_{2,n}(t)$, thus is dependent on $n$ not $T$. Whilst it is different from  the global penalty $\beta$, it plays an essential role in determining the value of $\beta$.

 The above probabilities $p_j(\gamma, n)$ are for the worst case over possible parameters of the data generating mechanism and time-points such that the specified region has no change or one change in the middle. In most situations, for example $x_{1:T}$ are independent and identically distributed, the probabilities will be the same for all choices. Note that by considering the maximum over the set $\{\bS_{2,n}(t)\}$, we must have $m^\ast\ge 1$, that is, there exists at least one true changepoint.

Moreover, we need another condition on the cost function if we do not fit a change near a true changepoint, which is given below.
\begin{condition} \label{cond.2}
Let $\bS_{\Delta,n}(t)=x_{t+1: t+2n}$ be a segment between $t+1$ and $t+2n$ that has a single changepoint which is at $t+n$,  and, with $\Delta=\mbox{dist}(\btheta_{(1)},\btheta_{(2)})$. Here $\btheta_{(1)}$ and $\btheta_{(2)}$ are the parameters associated with the segments immediately before and after $t+n$, and $\mbox{dist}(\cdot,\cdot)$ is a suitable measure of distance in parameter space that may be differing across applications. Then 
we have 
$$
\max_{t}\, \p\bigg(\Loss(\bS_{\Delta,n}(t),\varnothing)-\Loss^\ast(\bS_{\Delta,n}(t))\le z\bigg)\le p_5\big(\sigst(\Delta,n),z\big),
$$
where $\sigst(\Delta,n)$ is a function of signal strength, and we require  $p_5(y,z)\ra 0$ as $y\ra \infty$ and  $y/z\ra \infty$, and the probability is with-respect to the data generating mechanism.
\end{condition}
Condition \ref{cond.2} indirectly defines $\sigst(\Delta,n)$ as the signal strength of a change from segment parameter $\btheta_{(1)}$ to $\btheta_{(2)}$ with data of length $n$ on either side of the change point, where $\Delta$ is some appropriate measure of distance between the two segment parameters. The idea is that the reduction in cost of not fitting the change will be of the order of this signal strength. Again we bound the worst-case probability, but in many cases the probability will be the same for all segments in the set of $\{\bS_{\Delta,n}(t)\}$.

\subsection{Global changepoint consistency}\label{subsec: global.cons}

In this section, based on the introduced local region conditions, we establish consistency of estimates of the number and locations for the changepoints under the penalised cost approach when applied to data simulated with a general number $m^\ast$ of changes. The result is not limited to the type of changes and the underlying data generating mechanism. Thus it builds a general framework to obtain consistency theories for a broad class of changepoint problems.

As earlier, we denote the location of 
the changes by $\tau^\ast_{1:m^\ast}$, and we will denote the true segment lengths by $$\delta_j=\tau^\ast_{j}-\tau^\ast_{j-1}, \quad j=1,\ldots,m^\ast+1,$$ with, as before, $\tau^\ast_0=0$ and $\tau^\ast_{m^\ast+1}=T$.  Let the size of each change be denoted
by 
$$\Delta_j=\mbox{dist}\big(\btheta^\ast_{j+1}, \btheta^\ast_{j}\big), \quad j=1,\ldots,m^\ast,$$ where the distance $||\cdot||$ is defined in Condition \ref{cond.2}. In addition, we define $\delta_T=\min_j{\delta_j}$ and $\Delta_T=\min_j\Delta_j$.

Given a set of integers $n_{1:m^\ast}$ satisfying $0<n_j \le \min\{\delta_j,\delta_{j+1}\}$, we can partition the data into $2m^\ast+1$ regions $\{\bS_1,\bS_2, \dots,\bS_{2m^\ast+1}\}$, such that: 
\begin{equation}\label{eq:m-split}
\begin{cases}
\bS_{2j+1}=x_{(\tau^\ast_j+n_j+1):(\tau^\ast_{j+1}-n_{j+1})}, \\ 
\bS_{2j}=x_{(\tau^\ast_j-n_j+1):( \tau^\ast_j+n_j)},
\end{cases}
\end{equation}
where we define $n_0=n_{m^\ast+1}=0$. See the top plot of Figure \ref{Fig:1} for an example of this partitioning of the data. In this way, each region $\bS_j$ with an odd index $j$ does not contain a true changepoint, and each region with an even index has exactly one true changepoint in the middle, satisfying our definition of local regions. Therefore we can verify if Conditions~\ref{cond.1} and \ref{cond.2} are satisfied on $\{\bS_1,\bS_2, \dots,\bS_{2m^\ast+1}\}$, depending on specific problem.

 \begin{figure}
 \centering  \hspace*{-0.8cm}\includegraphics[scale=0.8]{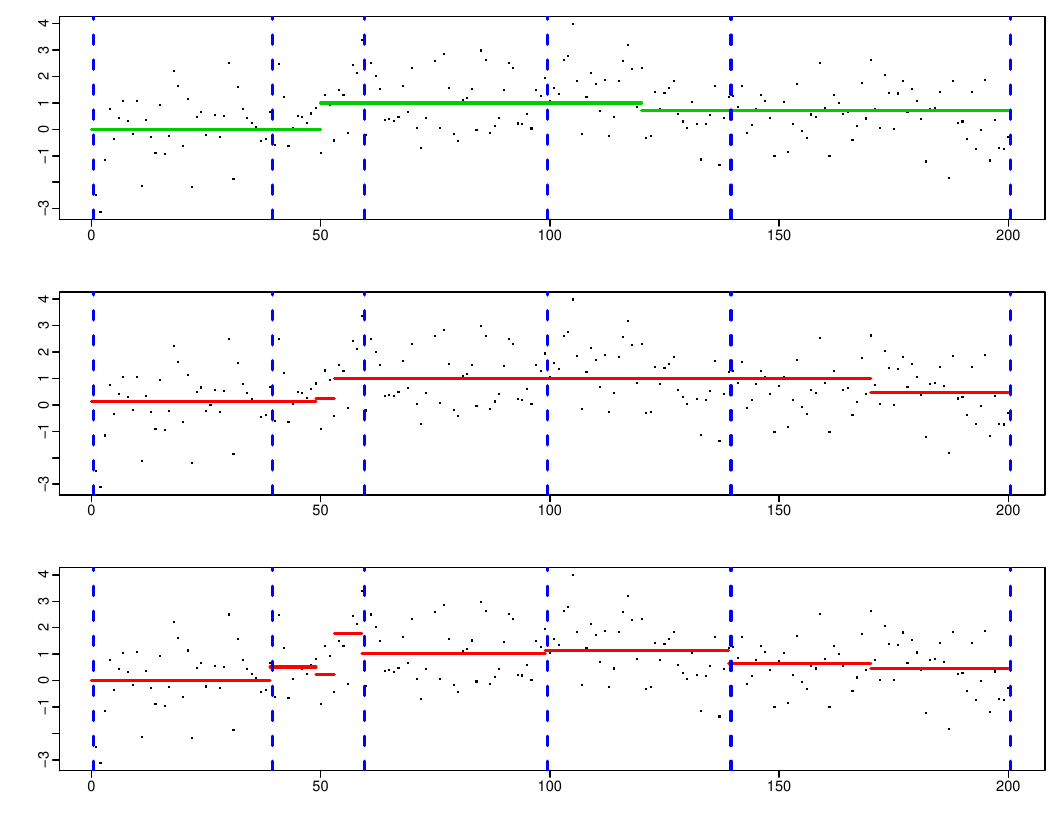}
  \caption{\label{Fig:1} \scriptsize{Top: Example partitioning of the data, for the univariate change-in-mean problem. The green-line is the true mean, which has two changes. Thus our partition has 5 regions, $\bS_{1:5}$, with the even regions containing a change and the odd-regions containing no change. The size of the even regions can be chosen based on the size of the change, with larger regions around smaller changes. Middle: a segmentation of the data, together with fitted mean (red-line), that violates the event (\ref{eq:event}). Such a segmentation will have errors within at least one region, in this case the fitted segmentation fits too many changes in $\bS_2$, misses the change in $\bS_4$ and erroneously fits a change in $\bS_5$. To show the penalised cost of such a segmentation will not be optimal, under our event $E_1$ (see the proof in Appendix) we bound the difference in the cost between such a segmentation and the true segmentation, by the difference of the cost if we fit the segmentation separately within each region (bottom figure) and the cost of the true segmentation with the true parameters (top figure). This difference is simply the sum of the differences of the fits in each region. The key idea is that for regions where a putative segmentation makes an error this difference will be sufficiently large that putative segmentation can not be optimal.}}
 \end{figure}

We are now ready to give a unified result for the changepoint estimation under our penalised cost criteria in the following theorem. 

\begin{thm}\label{thm.global}  Assume Conditions~\ref{cond.1} and \ref{cond.2} hold. For data $x_{1:T}$, let $\hat{m}$ and $\hat{\tau}_{1:\hat{m}}$ be the number and locations of changepoints we obtain by minimising the penalised cost (\ref{def.est}) with penalty $\beta$. 
For any $n_{1:m^*}$ where $0<n_j<\min\{\delta_j,\delta_{j+1}\}$,  
define an event $\mathcal{E}(\beta, n_{1:m^\ast})$ as
\begin{equation} \label{eq:event}
\mathcal{E}(\beta, n_{1:m^\ast}) = \left\{\hat{m}=m^\ast \mbox{ and \,}    |\htau_j-\tau_j^\ast| \le n_j, j=1,\dots, m^\ast\right\},
\end{equation}
and a minimum signal strength as $\sigst=\min_{j=1}^{m^\ast} \sigst(\Delta_j,n_j)$.
Then if
$$ \beta\geq \max\{\gamma^{(1)}_T,\gamma^{(2)}_{\max_i n_j} \} \mbox{ and } a(\beta,T)>2m^*b(\beta,T),$$
we have
\begin{align}\lbl{eq.consistency}
\p\bigg(\mathcal{E}(\beta, n_{1:m^\ast}) \bigg)\nonumber
\ge& 1-(m^\ast+1)p_1\left(\beta,T\right)-(m^\ast+1) p_2(\beta, T)-m^\ast p_3\left(\beta, \max_j n_j\right)\\&-m^\ast p_4\left(\beta, \max_j n_j\right)-m^\ast p_5\left({\sigst}, {\beta+a(\beta,T)}\right).
\end{align}
\end{thm}

The proof of Theorem \ref{thm.global}  is deferred to the appendix. Here we give a pictorial outline of the proof, using change-in-mean problem as an example in Figure \ref{Fig:1}.

Theorem \ref{thm.global} provides the probability bound on the event that the estimated number of changes is correct and the estimated location of each change $\tau^\ast_j$ is within the accuracy of $n_j$. The global penalty $\beta$ is chosen by collecting the maximum of $\gamma_{T}^{(1)}$ and $\gamma_{n_j}^{(2)}$ for all $1\le j\le m^\ast+1$.  If we specify any asymptotic regime such that  
$$\dfrac{\sigst}{\beta+a(\beta, T) }\rightarrow \infty, \quad \mbox{as }T\rightarrow \infty.$$ 
Then if $m^\ast$ is finite or diverging with $T$ in a slow rate, we have
$$m^\ast \min\left\{p_1(\beta,T), p_2(\beta,T),  p_3(\beta,\max_j n_j), p_4(\beta,\max_j n_j), p_5\big(\sigst, {\beta+a(\beta,T)}\big) \right\} \ra 0.$$ 
Hence such an event will hold with probability going to 1.

In the event such that $m^\ast=0$, we do not need to split the data into local regions but treat $x_{1:T}$ as in scenario (i) in Condition \ref{cond.1}; therefore, if $\beta\ge \gamma_{T}^{(1)}$, we have (\ref{eq.consistency}) still holds.
Also, it is simple to adapt the proof to show that we can replace $\beta>\gamma_T$ with $\beta>\gamma_{n_{\max}}$, where $n_{\max}=\max_j\{\delta_j\}$ is the maximum true segment length. This suggests the possibility of using smaller penalties in situations where the maximum segment length is known and is much shorter than $T$. For a given value of the minimum signal strength, $\sigst$, we can optimise the choice of $n_{1:m^*}$ that bound the accuracy of our estimates of the locations of each changepoint and make it to be tighter. Specifically we can choose $n_j$ to be the smallest value such that  $\sigst(\Delta_j,n_j)\geq\sigst$ for each $j=1,\dots, m^\ast$.

\section{Applications} \lbl{sec:application}

In this section, to show the usefulness and broad applicability of the general framework developed in Section \ref{sec:method},  we apply it to the estimation of change-in-mean, change-in-slope as well as changepoint in spike and exponential decay problems. The obtained consistency results for change-in-mean is not new, but it is more flexible than that in \cite{yao1978estimating}, allowing the number of changes, $m^\ast$, to diverge with $T$. For change-in-slope, the results is stronger than previous consistency results, as it carefully specifies the value of the penalty, $\beta$, that ensures consistency. This provides very important information on how to select $\beta$ and make the methods more accessible to practitioners.  For spike and exponential decay, we provide the first consistency result of this problem.

Note that in this section we always assume $m^\ast\ge 1$, meaning that there exists at least one changepoint. Otherwise, we only need to verify scenario (i) of Condition \ref{cond.1} and the global changepoint consistency holds trivially in all the three problems. Proofs for results in this section are deferred to the appendix of this paper.

\subsection{Change-in-mean problem}\lbl{sec:change-in-mean}
First, we revisit the canonical problem of detecting change-in-mean. Suppose we observe data $x_{1:T}$ with underlying decomposition,
$
x_{t}=\mu_t+\eps_t,
$
where $\eps_t\sim \cN(0,\sigma^2)$ are independent and identical distributed (i.i.d.) Gaussian, and $\mu_t$ are piecewise constant means, i.e.:
\begin{align*}
\mu_t=\theta^\ast_{j}, \quad \mbox{ if }\tau^\ast_{j-1}+1\le t\le \tau^\ast_{j}, \quad\mbox{ for all } j=1,\dots, m^\ast+1.
\end{align*} 

\noindent To estimate the set of changepoints, we use the square error loss as the cost function to measure fit to the data. 
That is, fitting a set of points $x_{s:e}$ with the same segment parameter, $\theta$, has cost function in the following form:
\begin{align}
\mC(x, {s, e};\,\theta)=\sum_{t=s}^e\dfrac{(x_t-\theta)^2}{\sigma^2}.
\end{align}
Note that in this application, no constraint is imposed on the parameters that minimise the cost function. Therefore, in fact we wish to minimise over $m$ and $\tau_{1:m}$, for the following penalised cost:
\begin{align}
\sum_{j=1}^{m+1}\sum_{t=\tau_{j-1}+1}^{\tau_j} \dfrac{(x_t-\bar{x}_{\tau_{j-1}+1: \tau_j})^2}{\sigma^2}+m\beta,
 \end{align}
where $ \bar{x}_{\tau_{j-1}+1:\tau_j}=\sum_{t=\tau_{j-1}+1}^{\tau_j}x_t/(\tau_j-\tau_{j-1})$. 

The above minimisation assumes knowledge of the noise variance $\sigma^2$, which can be regarded as a nuisance parameter. In practice if the variance is unknown, it is a common procedure in the literature to plug in, for example,  the Median Absolute Deviation (MAD) estimator \citep{hampel1974influence} applied to the differenced data \citep{baranowski2016narrow}, without any effect on the correctness of the theoretical analysis below. More specifically, for this change-in-mean exmaple we can set 
$$\hat\sigma = \dfrac{\mbox{median}\{|x_2-x_1|, \dots, |x_T-x_{T-1}|\}}{\sqrt{2}\Phi^{-1}({3}/{4})},$$
where $\Phi^{-1}(\cdot)$ is the quantile function of the standard normal distribution.

It is also possible to include the effect of the consistent estimation of variance parameter in the theoretical analysis. To this end, a simple approach is to use some sample-splitting procedure that will be helpful to avoid dependency in the analysis and simplify the results. For example, using the odd data to estimate the variance, and use even data to estimate the changepoints. A drawback of this method is we may increase the variance of our changepoint estimation.

As in Section \ref{subsec: global.cons}, we define the size of the change at $j$-th true changepoint $\tau^\ast_j$ as $$\Delta_j=\mbox{dist}(\theta^\ast_{j+1}, \theta^\ast_j)=|\theta^\ast_{j+1}-\theta^\ast_j|, \quad j=1,\dots,m^\ast,$$ which is the absolute mean difference in two consecutive segments.
The following propositions show Conditions \ref{cond.1} and \ref{cond.2} are satisfied for this change-in-mean application. 


\begin{proposition}\lbl{prop:mean1} 
Consider the following choices of  $\gamma_{n}^{(1)}$ and $\gamma_{n}^{(2)}$

\begin{align*}
\gamma_{n}^{(1)}&=\max\left\{(2+\epsilon)\log n, \,2\log n+\kappa_{1,1}\sqrt{\log n}, \, 2\log n+\kappa_{1,2}m^\ast\right\},\\
\gamma_{n}^{(2)}&=\max\{\kappa_{1,3}m^\ast\log(2n),  2\log(2n)+\kappa_{1,4}m^\ast\},
\end{align*} 

for some large enough constants $\kappa_{1,1}, \kappa_{1,2}, \kappa_{1,3}, \kappa_{1,4}$, and $\epsilon$ is an  positive constant that can be arbitrarily small. Moreover, let  
\begin{align}
a(\gamma,n)=\dfrac{\gamma-2\log n}{4} \quad \mbox{ and }  b(\gamma, n)=\dfrac{\gamma-2\log n}{4(2m^\ast+1)}.
\end{align}
We have the Conditions \ref{cond.1} are satisfied with  
\begin{align*}
p_1(\gamma, n)=2\exp\left(-\dfrac{\gamma-2\log n}{4}\right),\quad  &p_2(\gamma,n)=\exp\left(-\dfrac{\gamma-2\log n}{16(2m^\ast+1)}\right),\\
p_3(\gamma,n)=\exp\left(-\dfrac{\gamma-8\log (2n)}{4}\right), \quad &p_4(\gamma, n)=\exp\left(-\dfrac{\gamma-(8m^\ast+6)\log (2n)}{16(2m^\ast+1)}\right).
\end{align*}
\end{proposition}

\begin{proposition}\lbl{prop:mean2} Let $\mS(\Delta, n)=n\Delta^2/2$, we have Condition \ref{cond.2} is satisfied if $\mS(\Delta, n)/4\ge z\ge 5$, with $p_5\big({\mS(\Delta,n)},{z}\big)=2\exp\left(-z/20\right).$ 
\end{proposition}

 In this application the data mechanism is uniform at all time-point, therefore the probability is the same for all elements in $\{\bS_{1,n}(t)\}$, $\{\bS_{2,n}(t)\}$ and $\{\bS_{\Delta,n}(t)\}$. Here we remark that the above propositions illustrate as an example that Conditions \ref{cond.1} and \ref{cond.2} are satisfied for the change-in-mean problem,  where the constants in $\gamma_n^{(i)}, a(\gamma,n)$ and $p_j(\gamma,n)$, where $i=1,2$ and $j=1,\dots,5$, are not optimised. Note that by Proposition \ref{prop:mean1} we always have $a(\gamma, n)> 2m^\ast b(\gamma, n)$. 

\begin{thm}\lbl{thm:mean}  If $\beta=(2+\epsilon)\log T$, where $\epsilon>0$ is an arbitrarily small constant; and if 
\begin{align}
\delta_T\Delta_T^2\ge (16+10\epsilon)\log T \mbox{ and }
\Delta_T^2\ge \dfrac{(32+20\epsilon)\log T}{T^{1/(4m^\ast+3)}},
\end{align}
then for large enough $T$, 
 with probability at least $1-(7m^\ast+3) T^{-\epsilon/(32m^\ast+16)}$, 
\begin{align}\lbl{eq:consist.mean}
&\hat{m}=m^\ast,\max\limits_{j=1,\dots,m^\ast}|\htau_j-\tau_j|\Delta_j^2\le (16+10\epsilon)\log T.
\end{align}
\end{thm}

This result can be viewed as a finite-sample version of the existing consistency results for the change-in-mean problem \citep{yao1978estimating, tickle2018parallelisation}. In a recent work, \cite{wang2018optimal} provides a similar finite-sample result for change-in-mean problem and shows that the localisation rate in (\ref{eq:consist.mean}) is minimax optimal. In comparison, the extra condition on $\Delta_T$ and the inclusion of $m^\ast$ in the convergence rate in our Theorem \ref{thm:mean} is due to the fact that we specify the value of penalty to be $2+\epsilon$, other than a lack-of-track large constant, times $\log T$. The following corollary gives consistency for a specific asymptotic regime.

\begin{corollary}\lbl{coro:mean}
If $\beta=(2+\epsilon)\log T$, where $\epsilon>0$ is an arbitrarily small constant, assume that  $m^\ast=o\big(\log T\big)$, $\delta_T\Delta_T^2\ge c_1\log T$ and $\Delta_T^2\ge c_2 T^{-1/(4m^\ast+3)} \log T $,  we have  $$\p\bigg(\hat{m}=m^\ast,\max\limits_{j=1,\dots,m^\ast} |\htau_j-\tau_j|\Delta_j^2\le c_1\log T\bigg)\ra 1$$ 
as $T\ra \infty$, where $c_1$ and  $c_2$  are absolute constants that only depend on $\epsilon$.
\end{corollary}

For the standard in-fill asymptotic regime with a fixed number of true changes $m^\ast$ and constant size of minimum jump in the signals, i.e., $\Delta_T=O(1)$, from above corollary we can obtain a changepoint localisation rate of $O_p(\log T)$, which is the same order as in the classic results \citep{yao1978estimating}. Note that this order on accuracy of the changepoint locations could be further improved to basically $O_p(1)$ using the argument in \cite{yao1989least}. Also, similar improved accurary of consistent estimates is given in \cite{boysen2007scale} and \cite{boysen2009consistencies} for piecewise constant least
squares regression and general class of functions based on the solution to an $L_0$
least squares (Potts functional), under the assumption of the piecewise-constant mean function, fixed number of change points $m$, constant size of the minimum jump in the signal $\Delta_T$ and minimum segment length $\delta_T$ is of order $T$.

In this paper, we have a slight generalisation of these existing consistency results as we allow all the parameters $m$, $\delta_T$ and $\Delta_T$ to change
with the sample size $T$. Our localisation rate also mathches the minimax optimal rate. Note that some extra restrictions on $m^\ast$ (only slowly diverging with $T$) and $\Delta_T$ (lower bounded) is posed in our results due to the use of the argument of constructing global consistency from local region conditions, where in \cite{wang2018optimal} and \cite{Verzelen_2020} they are actually not required.

Finally, we remark that  similar results to Theorem \ref{thm:mean} and Corollary \ref{coro:mean} could be attained if we relax the Gaussianity assumption on noise to having an exponential or lighter tail. This can be done by adapting the chi-square bound in Lemma \ref{lemma.chi-square} to sub-gamma concentration inequalities, see similar argument in \cite{baranowski2016narrow}.  By taking a robust cost function and penalisation term $\beta$, the general framework should also work for the mean plus heavy-tailed noises model as considered in \cite{fearnhead2019changepoint}.

\subsection{Change-in-slope problem}\lbl{sec:change-in-slope}

For the change-in-slope application, we have the following decomposition of observations:
\begin{align}
x_t=f_t+\eps_t, \quad 1\le t\le T
\end{align}
 where $\eps_t\sim \cN(0,\sigma^2)$ are i.i.d Gaussian noises, and $f_t$ denote the piecewise linear mean signals, that is, for $j=1,\dots,m^\ast+1$:
 \begin{align}
 f_t=\theta^\ast_{j-1}+\dfrac{\theta^\ast_{j}-\theta^\ast_{j-1}}{\tau^\ast_{j}-\tau^\ast_{j-1}}(t-\tau^\ast_{j-1}),\quad \tau^\ast_{j-1}+1\le t\le \tau^\ast_j,
 \end{align}
In the above parameterisation, $\theta_{0:m^\ast+1}^\ast$ are values of the linear function at the changes $\tau_{0:m^\ast+1}$. As a consequence, this means we directly introduce the continuity constraint that enforces the value at the end of one segment to be equal to the value at the start of next segment.

We take the negative log-likelihood as the cost function, thus to fit a set of points $x_{s:e}$ such that  $\tau_{j-1}+1\le s< e\le \tau_j$, with the same bivariate structure parameter $\btheta=(\theta^{(1)},\theta^{(2)})$, the cost function is as follows:
\begin{align}
\mC(x, {s, e};\,\btheta)=\sum_{t=s}^e\dfrac{1}{\sigma^2}\left[x_t-\theta^{(2)}-\dfrac{\theta^{(1)}-\theta^{(2)}}{\tau_{j}-\tau_{j-1}}(t-\tau_{j-1})\right]^2.
\end{align}

\noindent Due to requirement of of continuity at the changepoint, we need a constrained minimisation on the corresponding overall cost function for fitting a a set of changes $\tau_{1:m}$, i.e., $\theta_j^{(2)}=\theta_{j-1}^{(1)}$ for all $j=1,\dots,m$ and $\btheta_j=\left(\theta_j^{(1)},\theta_j^{(2)}\right)$. Therefore the overall cost function takes the following form:
\begin{align}
\Loss(x_{1:T},\tau_{1:m})=&\min_{\btheta_{1:m+1}, \theta_j^{(2)}=\theta_{j+1}^{(1)}}\sum_{j=1}^{m+1}\mC\left(x_{\tau_{j-1}+1:\tau_j}; \btheta_j\right)\nonumber\\
=&\min_{\theta_{0:m+1}}\sum_{j=1}^{m+1} \sum_{t=\tau_{j-1}+1}^{\tau_{j}}\dfrac{1}{\sigma^2}\left[x_t-\theta_{j-1}-\dfrac{{\theta_{j}}-\theta_{j-1}}{\tau_{j}-\tau_{j-1}}(t-\tau_{j-1})\right]^2.
\end{align}

\noindent In this application, we define the size of change at the $j$-th changepoint, $\tau^\ast_j$, as
 \begin{align}
 \Delta_j=\mbox{dist}(\btheta^\ast_{j+1}, \btheta^\ast_{j})=\bigg\lvert\dfrac{\theta^\ast_{j+1}-\theta^\ast_{j}}{\tau^\ast_{j+1}-\tau^\ast_j}-\dfrac{\theta^\ast_{j}-\theta^\ast_{j-1}}{\tau^\ast_{j}-\tau^\ast_{j-1}}\bigg\rvert,
\end{align}
which is the absolute difference of slopes in two consecutive segments. Again, the nuisance noise variance $\sigma^2$, if unknown, can be robustly estimated \citep{fearnhead2018detecting}, for example we can set:
$$
\hat\sigma=\dfrac{\mbox{Median}\{|x_1-2x_2+x_3|, \dots, |x_{T-2}-2x_{T-1}+x_T|\}}{\sqrt{6}\Phi^{-1}(3/4)}.
$$


In order to study the property of the changepoint detection for this application, we first need to verify Conditions \ref{cond.1} and \ref{cond.2} are satisfied, which is shown is the following propositions. 

\begin{proposition}\lbl{prop:slope1} 
Consider following choice of $\gamma_n^{(1)}$ and $\gamma_n^{(2)}$:
\begin{align*} 
 \gamma_{n}^{(1)}&=\max\left\{(2+\epsilon)\log n, 2\log n+\kappa_{2,1}\sqrt{\log n}, \, 2\log n+\kappa_{2,2}m^\ast\right\},\\
\gamma_n^{(2)}&=\max\bigg\{(3+\epsilon)\log (2n), 2\log(2n)+\kappa_{2,3}\log(\log(2n)),   2\log(2n)+\kappa_{2,4}m^\ast\bigg\},
\end{align*}
for some large enough constants $\kappa_{2,1}, \kappa_{2,2}, \kappa_{2,3}, \kappa_{2,4}$, and $\epsilon$ is an  positive constant that can be arbitrarily small. Moreover, let  
\begin{align}
a(\gamma,n)=\dfrac{\gamma-2\log n}{6} \quad\mbox{ and }\quad  b(\gamma, n)=\dfrac{\gamma-2\log n}{6(2m^\ast+1)}.
\end{align}
We consequently find that Conditions \ref{cond.1} is satisfied by:  
\begin{align*}
p_1(\gamma, n)=2\exp\left(-\dfrac{\gamma-2\log n}{6}\right), \quad &p_2(\gamma,n)=\exp\left(-\dfrac{\gamma-2\log n}{24(2m^\ast+1)}\right),\\  
p_3(\gamma,n)=\dfrac{9}{4}\exp\left(-\dfrac{\gamma-3\log (2n)}{3} \right), \quad &p_4(\gamma, n)=\exp\left(-\dfrac{\gamma-2\log (2n)}{24(2m^\ast+1)}\right).
\end{align*}

\end{proposition}

\begin{proposition}\lbl{prop:slope2} 
Let $\mS(\Delta, n)=n^3\Delta^2/25$, then Condition \ref{cond.2} is satisfied if $\mS(\Delta, n)/4\ge z\ge 8$ and $n\ge 2$, with $p_5\big({\mS(\Delta,n)},{z}\big)=2\exp\left(-z/20\right).$ 
\end{proposition}

Propositions \ref{prop:slope1} and \ref{prop:slope2} have strong similarity with Propositions \ref{prop:mean1} and \ref{prop:mean2} for the change-in-mean problem, where again the constants are not optimised for simplicity purposes.
Based on Theorem \ref{thm.global} we can obtain the following theorem.

\begin{thm}\lbl{thm:slope}  If $\beta=(2+\epsilon)\log T$, where $\epsilon>0$ is an arbitrarily small constant; and if 
\begin{align}
\delta_T^3\Delta_T^2\ge (200+350\epsilon/3)\log T \mbox{ and }
\Delta_T^2\ge \dfrac{(1600+2000\epsilon/3)\log T}{T^2},
\end{align}
then for large enough $T$, 
with probability at least $1-(33m^\ast/4+3) T^{-\epsilon/(48m^\ast+24)}$,
\begin{align}
&\hat{m}=m^\ast,\max\limits_{j=1,\dots,m^\ast}|\htau_j-\tau_j|^{3}\Delta_j^2\le (200+350\epsilon/3)\log T.
\end{align}
\end{thm}

In terms of asymptotics, we have the following corollary.
\begin{corollary}\lbl{coro:slope}
If $\beta=(2+\epsilon)\log T$, where $\epsilon>0$ is an arbitrarily small constant, assume that  $m^\ast=o\big(\log T\big)$, $\delta_T^3\Delta_T^2\ge c_3\log T$ and $\Delta_T^2\ge {c_4}\log T/{T^2}$,  we have  $$\p\bigg(\hat{m}=m^\ast,\max\limits_{j=1,\dots,m^\ast}|\htau_j-\tau_j|^3\Delta_j^2\le c_3\log T\bigg)\ra 1$$ 
as $T\ra \infty$, where $c_3$ and $c_4$ are a absolute constants that only depend on $\epsilon$.
\end{corollary}

For the standard in-fill asymptotic regime with a fixed number of true changes $m^\ast$, we would have $\Delta_T=O(T^{-1})$. In such case we get a bound on the error of location estimates that is just a logarithmic factor worse than the minimax rate of $T^{2/3}$ \citep{raimondo1998sharp}. The results also holds when $m^\ast$ is diverging at a rate slower than $\log T$. 

We emphasis that this result is stronger than previous consistency results of the change-in-slope model derived in \cite{baranowski2016narrow} and \cite{fearnhead2018detecting}, as it specifies the value of the penalty $\beta$ that ensures consistency. This is a non-trivial technical improvement as shown in the Appendix \ref{subsec:basis}. In this case to get tighter results for the choice of penalty, we need to take account of the positive dependency in the reduction of cost of similar segmentations. We need these tighter bounds because for the in-fill asymptotic regime the accuracy of estimating a change-in-slope is polynomial in $T$ rather than logarithmic in $T$. This accuracy impacts, and increases, the number of possible segmentations we can fit to data in our local region that have one changepoint in the middle location. We believed similar arguments can be applied to refine the value of penalty in other changepoint detection problems.


\subsection{Changepoint in spike and exponential decay problem}\lbl{sec:change-in-spike}
In this application, the observations $x_{1:T}$ have an underlying decomposition
$
x_t=c_t+\eps_t,
$
 where $\eps_t\sim \cN(0,\sigma^2)$ are i.i.d. Gaussian innovations, and the mean function, $c_t$, follows a piecewise spike and exponential decay model. That is, for $j=1,\dots, m^\ast+1$, 
$$c_t=\theta^\ast_j\alpha^{t-\tau^\ast_{j-1}-1} , \quad  \tau^\ast_{j-1}+1 \le t\le \tau^\ast_j,$$ where $0<\alpha< 1$ is the decay rate. When $\alpha=1$ this reduces to the change-in-mean problem in Section \ref{sec:change-in-mean}. 

We take the square error loss as the cost function, 
so to fit a set of points $x_{s:e}$ such that $\tau_{j-1}+1\le s< e\le \tau_j$ with the same parameter $\theta_j$,  we have the following cost:
\begin{align}
\mC(x, {s, e};\,\theta_j)=\sum_{t=s}^e\dfrac{1}{\sigma^2} \left(x_t-\theta_j\alpha^{t-\tau_{j-1}-1}\right)^2.
\end{align}

\noindent Therefore, the corresponding cost function is
\begin{align}\lbl{eq.loss.spike}
\Loss(x_{1:T},\tau_{1:m})=\min_{\theta_{j+1}\neq \theta_{j}\alpha^{\tau_j-\tau_{j-1}}}\sum_{j=1}^{m+1} \sum_{t=\tau_{j-1}+1}^{\tau_{j}}\dfrac{1}{\sigma^2}\left(x_t-\theta_j\alpha^{t-\tau_{j-1}-1}\right)^2,
\end{align}
where $\tau_0=0$ and $\tau_{m+1}=T$. In this application, we minimises over $m$, $\tau_{1:m}$ and $\theta_{1:m+1}$ in the penalised cost $\Loss(x_{1:T}; \tau_{1:m})+m\beta$ to estimate the number of changepoints and their positions.

 We define the size of $j$-th changepoint at $\tau^\ast_j$ as
\begin{align}
 \Delta_j=\mbox{dist}(\theta^\ast_{j+1}, \theta^\ast_j)=\vert\theta^\ast_{j+1}-\theta^\ast_{j}\alpha^{\tau^\ast_j-\tau^\ast_{j-1}-1}\rvert,
\end{align}
which is the size of the jump in the signal from the end of $j$-th segment to the beginning of the $(j+1)$-th segment. Note that solving the problem (\ref{eq.loss.spike}) requires knowledge of the decay rate $\alpha$ and the noise variance $\sigma^2$. For methods to estimate these see \cite{jewell2019fast} and \cite{jewell2018exact}.

In the following propositions, we show that local Conditions \ref{cond.1} and \ref{cond.2} are satisfied for this application.
\begin{proposition}\lbl{prop:spike1} 
Let 
\begin{align*}
\gamma_{n}^{(1)}&=\max\left\{(2+\epsilon)\log n, 2\log n+\kappa_{3,1}\sqrt{\log n},  2\log n+\kappa_{3,2}m^\ast\right\},\\
\gamma_{n}^{(2)}&=\max\bigg\{\kappa_{3,3}\log(2n),  2\log(2n)+\kappa_{3,4}m^\ast\bigg\},
\end{align*} 
for some large enough constants $\kappa_{3,1}, \kappa_{3,3}, \kappa_{3,3}, \kappa_{3,4}$, and $\epsilon$ is an  positive constant that can be arbitrarily small. Moreover, let  $a(\gamma,n)=\dfrac{\gamma-2\log n}{4}$ and  $b(\gamma, n)=\dfrac{\gamma-2\log n}{4(2m^\ast+1)}$.
We have that Conditions \ref{cond.1} is satisfied with  
\begin{align*}
p_1(\gamma, n)=2\exp\left(-\dfrac{\gamma-2\log n}{4}\right),\quad  &p_2(\gamma,n)=\exp\left(-\dfrac{\gamma-2\log n}{16(2m^\ast+1)}\right),\\
p_3(\gamma,n)=\exp\left(-\dfrac{\gamma-8\log (2n)}{4}\right), \quad &p_4(\gamma, n)=\exp\left(-\dfrac{\gamma-(8m^\ast+6)\log (2n)}{16(2m^\ast+1)}\right).
\end{align*}
\end{proposition}

\begin{proposition}\lbl{prop:spike2} Let $\mS(\Delta, n)=\dfrac{\Delta^2}{(1-\alpha^{2n})(1-\alpha^2)}$, then Condition \ref{cond.2} is satisfied if $\mS(\Delta, n)/4\ge z\ge 5$, with $$p_5\bigg({\mS(\Delta,n)},{z}\bigg)=2\exp\left(-\dfrac{z}{20}\right).$$ 
\end{proposition}

Propositions \ref{prop:spike1} and \ref{prop:spike2} is quite similar to Propositions \ref{prop:mean1} and \ref{prop:mean2} for the change-in-mean problem. The only difference lies in the form of signal strength. Based on Theorem \ref{thm.global} we can obtain the following theorem.

\begin{thm}\lbl{thm:spike}  If $\beta=(2+\epsilon)\log T$, where $\epsilon>0$ is an arbitrarily small constant; and if we have
\begin{align}\label{spike.require.1}
\dfrac{\Delta_T^2}{(1-\alpha^{2\delta_T})(1-\alpha^2)}\ge (8+5\epsilon)\log T
\end{align}
and 
\begin{align}\label{spike.require.2}
\log_{\alpha}\left(1-\dfrac{\Delta_T^2}{(1-\alpha^2)(8+5\epsilon)\log T}\right)\le T^{2/(8m^\ast+6+\epsilon)},
\end{align}  
 then for large enough $T$, 
with probability at least $1-(7m^\ast+3)T^{-\epsilon/(32m^\ast+16)}$
\begin{align}
\hat{m}=m, \min\limits_{j=1,\dots,m^\ast} \dfrac{\Delta_j^2}{\left(1-\alpha^2\right)\left(1-\alpha^{2|\htau_j-\tau^\ast_j|}\right)}\ge (8+5\epsilon)\log T  .
\end{align}
\end{thm}

In terms of asymptotic regime, note that the signal strength in this application $\mS(\Delta, n)=O(\Delta^2)$ will become finite if $\alpha$ is fixed. A more natural asymptotic regime is to assume that we are able to obtain data at a higher frequency, which would correspond to $\alpha \ra 1$ as $T \ra \infty$. 
\begin{corollary}\lbl{coro:spike}
If $\beta=(2+\epsilon)\log T$, where $\epsilon>0$ is an arbitrarily small constant. Assume that $c_5\log T\le \dfrac{\Delta_T^2}{(1-\alpha^{2\delta_T})(1-\alpha^2)}$, and there exists universal positive constant $D<1$ such that $\log_{\alpha}D \le T^{2/(8m^\ast+6+\epsilon)}$. Let
$m^\ast=o\big(\log T)$ and $1-\alpha\ge c_6(\Delta_T^2/\log T)$, we have  $$\p\bigg(\hat{m}=m^\ast, \min\limits_{j=1,\dots,m^\ast} \dfrac{\Delta_j^2}{\left(1-\alpha^2\right)\left(1-\alpha^{2|\htau_j-\tau^\ast_j|}\right)}\ge c_{5}\log T\bigg)\ra 1$$ 
as $T\ra \infty$, where $c_5$ and $c_{6}$ are absolute constants only depends on $\epsilon$.
\end{corollary}

For the standard in-fill asymptotic regime with a fixed number of true changes $m^\ast$, we would have $\Delta_T=O(1)$. In such case if $\alpha=\exp(-c/T)$ for some positive constant $c$, we get a bound on the error of changepoint location estimate, $\max_j|\htau_j-\tau^\ast_j|$, of order not greater than $\log T$ with probability going to $1$.

\section{Discussion}\label{sec:discussion}

In this paper our key motivation is to show that for a class of changepoint methods, statistical properties for detecting multiple changepoints using penalised cost approaches can be derived from the behaviour of the method when analysing data with either no changepoint or a single changepoint, i.e., on the so-called local region. Therefore, we propose a general framework to prove consistency results of a broad class of penalised cost approaches for detecting multiple changepoints, only assuming the properties on the local regions.

These  properties on the local regions are often easier to verify; for example the results for the three applications we considered in Section \ref{sec:application} almost all follow from bounds on chi-squared random variables and the use of a Bonferroni correction. The one exception is for the change-in-slope model where to get sharper results on the choice of penalty, we need to consider the dependency between the cost of fitting similar segmentations. The related techniques are presented in Appendix \ref{subsec:basis}, which we believe to be of independent interest.

Our focus on the theoretical aspects of different types of change, such as of change-in-slope and spike plus exponential decay problem, without carrying out numerical studies is due to the fact that fast computational algorithms based on dynamic programming and (functional) pruning have been well-developed, and their numerical performance and applications to real-world data have been extensively analysed in the literature. Using the technique of local region conditions, this paper closes the theoretical gap in these applications. More importantly, the technique provides a powerful tool to derive consistency results for many other multiple changepoint applications.

In additiona to the applications considered in Section \ref{sec:application}, the general framework can be readily applied to more complicated cases, for example, the signal within the segments is non-linear \citep{yu2022localising},  structure changes is in the moments or quantiles \citep{fisch2022linear,baranowski2016narrow}, and other structure changes such as and autocovariance in time series. Moreover, as mentioned earlier the gaussianity condition on the noise can be relaxed as well \citep{fearnhead2019changepoint}.

The theoretical results given in this paper could be further improved in two regards. First the constants we obtained have not been optimised. More importantly, as mentioned earlier the results on accuracy of the changepoint locations could be improved using arguments similar to \cite{yao1989least} and \cite{}. The idea is to leverage the result that shows we accurately estimate the location of $\tau_{j-1}$ and $\tau_{j+1}$ and then show that the error in estimating $\tau_j$ converges in distribution to the accuracy of estimating the location of a single changepoint from data in the region between $\tau_{j-1}$ and $\tau_{j+1}$.

\appendix

\section{Proofs of Section 2}\label{app}

\textit {Proof of Lemma \ref{lem:1}.} 
It is trivial to verify (\ref{lem.property.2}) as once we fix the segment parameters the cost, $\Loss^\ast(\cdot)$, is additive over data-points.  

Property  (\ref{lem.property.1}) follows because we are minimising the cost on the left-hand side over a more constrained space for the segment parameters. If $r$ is not a changepoint location, then
\begin{align*} 
&\Loss(x_{s:r};\tau_{u:v})+\Loss(x_{r+1:e};\tau_{u:v})\\
=&\min_{\btheta_{u:k-1},\btheta_k^{(1)}}\left(\sum_{j=u}^{k-1}\mC\left(x_{\tau_{j-1}+1:\tau_j};\btheta_{j}\right)+\mC\left(x_{\tau_{k-1}+1:r};\btheta^{(1)}_{k}\right)\right)\\
&+\min_{\btheta_{k}^{(2)}, \btheta_{k+1:v}}\left(\mC\left(x_{r+1:\tau_k};\btheta_{k}^{(2)}\right)+\sum_{j=k+1}^v\mC\left(x_{\tau_{j-1}+1:\tau_j};\btheta_{j}\right)\right)
\\
\le &\min_{\btheta_{u:v}} \Biggl(\sum_{j=u}^{v}\mC\left(x_{\tau_{j-1}+1:\tau_j};\btheta_{j}\right) \Biggr)\\
=& \Loss(x_{s:e};\tau_{u:v}),
\end{align*}
where all minimisations include any constraints on segment parameters in neighbouring segments. The inequality comes from the fact that for $\Loss(x_{s:r};\tau_{u:v})+\Loss(x_{r+1:e};\tau_{u:v})$ we have no constraint between $\btheta_k^{(1)}$ and $\btheta_k^{(2)}$, but for $\Loss(x_{s:e};\tau_{u:v})$ we have $\btheta_k^{(1)}=\btheta_k^{(2)}:=\btheta_k$. A similar argument applies if $r$ is a changepoint, as $\Loss(x_{s:r};\tau_{u:v})+\Loss(x_{r+1:e};\tau_{u:v})$ will not apply any constraint between the segment parameters for the segments immediately before and after $r$. \qed

\medskip

\noindent\textit{Proof of Theorem \ref{thm.global}.}
Consider a split of the data as (\ref{eq:m-split}) for the specified $n_{1:m^\ast}$. Let $l_j$ be the number of data in region $\bS_j$, therefore $l_{2k+1}=\delta_k-n_k-n_{k+1}$ and $l_{2k}=2n_k$. We define an event, $E_1$, based on this split, such that the following holds jointly for regions $\bS_{1:2m^\ast+1}$: \\
if $j$ is odd, 
$$
\min_{k\ge 1,\tau_{1:k}}\{\Loss\left(\bS_j;\tau_{1:k}\right)+k\beta\}-\Loss^\ast\left(\bS_j\right)\ge a\left(\beta, l_j\right), \mbox{ and }
$$
$$
\Loss^\ast(\bS_j)- \Loss(\bS_j;\varnothing) \leq b(\beta,l_j);
$$
if $j$ is even, 
$$\min_{k\ge 2,\tau_{1:k}}\{\Loss(\bS_j;\tau_{1:k})+(k-1)\beta\}-\Loss^\ast(\bS_j)\ge a(\beta,l_j), 
$$
$$\Loss^\ast(\bS_j)- \min_{x_\tau\in \bS_j} \Loss(\bS_j;\tau ) \leq b(\beta, l_j),\mbox{  and}$$
$$\Loss(\bS_j,\varnothing)-\Loss^\ast(\bS_j)\ge a(\beta, l_j)+\beta;
$$


In what follows, we will condition on this event holding. Since $\beta\ge \max_k\{\gamma^{(1)}_{l_{2k+1}}, \gamma^{(2)}_{l_{2k}}\}$, due to the Conditions  \ref{cond.1} and \ref{cond.2}, by a simple union bound we have the probability of this is lower bounded by 
\begin{align*}
1-(m^\ast+1)p_1(\beta,T)-(m^\ast+1) p_2(\beta, T)-m^\ast p_3(\beta, \max_j n_j)\\-m^\ast p_4(\beta, \max_j n_j)-m^\ast p_5\left(\sigst,\beta+a(\beta, T)\right).
\end{align*} 
Now for any segmentation $\tau_{1:m}$ we can compare the penalised cost of that segmentation with the penalised cost of the true segmentation. Trivially, we have
\begin{align*}
&\left\{\Loss(x_{1:T};\tau^\ast_{1:m^\ast})+\beta m^\ast\right\} - \left\{\Loss(x_{1:T};\tau_{1:m})+\beta m\right\} \\ 
&\leq \Loss^*(x_{1:T})-\Loss(x_{1:T};\tau_{1:m}) +(m^\ast-m)\beta,
\end{align*}
as $\Loss(x_{1:T};\tau^\ast_{1:m^\ast})$ minimises over the segment parameters whereas $\Loss^*(x_{1:T})$ fixes them to their true values. Using Lemma \ref{lem:1} we have
$$
\Loss^*(x_{1:T})-\Loss(x_{1:T};\tau_{1:m}) \leq \sum_{j=1}^{2m^*+1} \left\{ \Loss^*(\bS_j)-\Loss(\bS_j;\tau_{1:m}) 
\right\}.
$$
We can partition the set of regions into $\mathcal{A}$, $\mathcal{B}$ and $\mathcal{D}$, which are defined as the regions where the putative segmentation, $\tau_{1:m}$, fits too many changes, too few changes, or the correct number of changes. Let $k_j$ and $k^*_j$, respectively, denote the number of changepoints in region $j$ in the putative and the true segmentation.
Thus
\begin{align*}
&\left\{\Loss(x_{1:T};\tau^\ast_{1:m^\ast})+\beta m^\ast\right\} - \left\{\Loss(x_{1:T};\tau_{1:m})+\beta m\right\}\\
 \leq&\sum_{j \in \mathcal{A} } \left\{ \Loss^*(\bS_j)-\Loss(\bS_j;\tau_{1:m})   +(k^*_j-k_j)\beta  \right\} \\ 
 &+ \sum_{j \in \mathcal{B} } \left\{ \Loss^*(\bS_j)-\Loss(\bS_j;\tau_{1:m}) +(k^*_j-k_j)\beta  \right\} \\
 &+ \sum_{j \in \mathcal{D} } \left\{ \Loss^*(\bS_j)-\Loss(\bS_j;\tau_{1:m}) +(k^*_j-k_j)\beta  \right\}
\end{align*}
Conditional on our event $E_1$ holding, terms in the first two sums can be bounded above by $-a(\beta,T)$, while terms in the final sum can be bounded above by $b(\beta,T)$. If we let $m_D$ denote the number of terms in the final sum we have
\begin{align*}
&\left\{\Loss(x_{1:T};\tau^\ast_{1:m^\ast})+\beta m^\ast\right\} - \left\{\Loss(x_{1:T};\tau_{1:m})+\beta m\right\}\\
\leq& m_{D}b(\beta,T)-(2m^*+1-m_{D})a(\beta,T).
\end{align*}
For any segmentation $\tau_{1:m}$ with which $\mathcal{E}$ does not hold when $\hat{m}=m$ and $\hat{\tau}_{1:\hat{m}}=\tau_{1:m}$ we have $m_D\leq 2m^*$. Thus as $a(\beta,T)>2m^*b(\beta,T)$ the right-hand side of this equality will be strictly less than 0. Hence, conditional on event $E_1$ holding, no such segmentation can minimise our penalised cost. \qed

\section{Proofs of Section~\ref{sec:change-in-mean}}\label{suppsec:change-in-mean}

Lemma \ref{lemma.chi-square} below is a direct adaptation of  Lemma 1 in \cite{laurent2000adaptive} and Lemma 8.1 in \cite{birge2001alternative}. We will use it repeatedly.

\begin{lemma}\label{lemma.chi-square}
Let $\chi^2_k$ be a central chi-square statistic with $k$ degrees of freedom and $\chi^2_k(\nu)$ a chi-square statistic with $k$ degrees of freedom and non-centrality parameter $\nu$. For any $x> k$  and $y< k+\nu$, we have
\begin{align}
\p(\chi^2_k\ge x)\le \exp\left(-\dfrac{x-\sqrt{k(2x-k)}}{2}\right),\lbl{ineq.chisq.upper}\\
\p(\chi^2_k(\nu)\le y)\le \exp\left(-\dfrac{(k+\nu-y)^2}{4k+8\nu}\right).\lbl{ineq.chisq.lower}
\end{align}
\end{lemma}

\medskip

\subsection{Proof of Propositions~\ref{prop:mean1} and \ref{prop:mean2}}

\begin{lemma}\label{lemma.addk.A} For any $t$ and any model $\cN_n(t)$, if $\bS=x_{t+1:t+n}$ is a region that contains no true changepoint, then for any $\gamma\ge \gamma_n^{(1)}$, where 
\begin{equation}
\gamma_n^{(1)}=\max\left\{(2+\epsilon)\log n, 2\log n+8\sqrt{16+2\log n}+32, 2\log n + 32(2m^\ast + 1)\right\}
\end{equation}
with $\epsilon>0$ is a constant, we have 
 $$\p\left(\min_{1\le k,\tau_{1:k}}\{\Loss\left(\bS;\tau_{1:k}\right)+k\gamma\}-\Loss^\ast\left(\bS\right)\le \dfrac{\gamma-2\log n}{4}\right)< 2\exp\left(-\dfrac{\gamma-2\log n}{4}\right),$$
  and 
  $$
\p\bigg( \Loss^\ast(\bS)- \Loss(\bS;\varnothing) \geq \dfrac{\gamma-2\log n}{4(2m^\ast+1)} \bigg) \le \exp\left(-\dfrac{\gamma-2\log n}{16(2m^\ast+1)}\right).
$$
\end{lemma}

\begin{proof}
For any $t$ and model $\cN_n(t)$  such that $\bS=x_{t+1:t+n}$ is a region that contains no true changepoint, since $\eps_{t+1:t+n}$ are i.i.d, it is easy to derive that 
$
\Loss^\ast\left(\bS\right)-\Loss({\bS};\tau_{1:k}) \sim \chi^2_{k+1}
$ and 
$\Loss^\ast\left(\bS\right)-\Loss({\bS};\varnothing)\sim \chi^2_{1}$.

Letting $a(\gamma,n)=(\gamma-2\log n)/4$, we have
\begin{align*}
&\p\left(\min_{k,\tau_{1:k}}\{\Loss(\bS; \tau_{1:k})+k\gamma-\Loss^{\ast}(\bS)\}\le a(\gamma,n) \right)\\
&\le \sum_{k=1}^n \p\left(\max_{\tau_{1:k}}\left[\Loss^{\ast}({\bS})-\Loss({\bS};\tau_{1:k})\right]\ge k\gamma-a(\gamma,n)\right)\\
&\le \sum_{k=1}^\infty {n \choose k} \p\left( \chi^2_{k+1}\ge k\gamma-a(\gamma,n)\right)
\end{align*}
Combined with (\ref{ineq.chisq.upper}) in Lemma \ref{lemma.chi-square},  
\begin{align*}
& {n \choose k} \p\left( \chi^2_{k+1}\ge k\gamma-a(\gamma,n)\right)\\
<& \dfrac{n^k}{k!}   \exp\left(-\dfrac{k\gamma-a(\gamma,n)-\sqrt{(k+1)(2k\gamma-2a(\gamma,n)-k-1)}}{2}\right)\\
&<\dfrac{1}{k!}   \exp\left(-\dfrac{(k-1/4)(\gamma-2\log n)-\sqrt{2k(k+1)\gamma}}{2}\right)
\end{align*}
As long as $\gamma\ge 2\log n+8\sqrt{16+2\log n}+32$, we have 
\begin{align}
\sqrt{2k(k+1)\gamma}\le \dfrac{k(\gamma-2\log n)}{4}.
\end{align}
Hence 
\begin{align*}
&\p\left(\min_{k,\tau_{1:k}}\{\Loss(\bS; \tau_{1:k})+k\gamma-\Loss^{\ast}(\bS)\}\le a(\gamma,n) \right)\\
&< \sum_{k=1}^{\infty}\dfrac{1}{k!} \exp\left(-\dfrac{(3k/4-1/4)(\gamma-2\log n)}{2}\right)\\
&<\sum_{k=1}^{\infty}\dfrac{1}{2^{k-1}} \exp\left(-\dfrac{\gamma-2\log n}{4}\right)\\
&<2\exp\left(-\dfrac{\gamma-2\log n}{4}\right).
\end{align*}
Next we prove the second inequality. For any specified $m^*$, define $b(\gamma,n)=(\gamma-2\log n)/(8m^*+4)$.
Since $\gamma>2\log n+32(2m^\ast+1)$, we have $b(\gamma, n)\ge 8$, which leads to $\sqrt{2b(\gamma,n)}\le b(\gamma,n)/2$. Applying (\ref{ineq.chisq.upper}) in Lemma \ref{lemma.chi-square}, we obtain
\begin{align*}
\p\bigg( \Loss^\ast(\bS)- \Loss(\bS;\varnothing) \geq b(\gamma, n) \bigg)=&\p\left(\chi^2_1\ge  b(\gamma, n) \right)\\
< &\exp\left(-\dfrac{b(\gamma,n)-\sqrt{2b(\gamma,n)}}{2}\right)\\
\le &\exp\left(-\dfrac{b(\gamma,n)}{4}\right)
\\= &\exp\left(-\dfrac{\gamma-2\log n}{16(2m^\ast+1)}\right).
\end{align*}
\end{proof}

\begin{lemma}\label{lemma.addk2.A} For any $t$ and any model, if $\bS=x_{t+1:t+2n}$ is a region that contains a single changepoint  at $\tau^\ast=t+n$, for any $\gamma\ge \gamma_{n}^{(2)}$, where
\begin{equation}
 \gamma_{n}^{(2)}=\max\left\{(8m^\ast+6+\epsilon)\log (2n), 2\log (2n)+64(2m^\ast+1)\right\},
 \end{equation}
  with $\epsilon>0$ is a constant,  we have  
 $$\p\left(\min_{k\ge 2,\tau_{1:k}}\{\Loss({\bS}; \tau_{1:k})+(k-1)\gamma\}- \Loss^{\ast
}(\bS)\le \dfrac{\gamma-2\log (2n)}{4}\right)<  \exp\left(-\dfrac{\gamma-8\log(2n)}{4}\right),$$
and 
$$
\p\left(\Loss^\ast({\bS})-\min_{\tau_1}\Loss(\bS; \tau_1)\ge \dfrac{\gamma-2\log (2n)}{4(2m^\ast+1)}\right) \le \exp\left(-\dfrac{\gamma-(8m^\ast+6)\log (2n)}{16(2m^\ast+1)}\right).
$$
\end{lemma}
\begin{proof}
Note that for any $\tau_{1:k}$ on $\bS$, 
$$
\Loss^\ast(\bS)-\Loss({\bS};\tau_{1:k})\le\Loss^\ast({\bS})-\Loss({\bS};\tau_{1:k},\tau^\ast)\sim\chi^2_{k+2}, \mbox{ and }
$$
$$
\Loss^\ast(\bS)-\Loss\left(\bS; \tau_1\right)\sim \chi^2_2.
$$

Let $a(\gamma, 2n)=\dfrac{\gamma-2\log(2n)}{4}\le \gamma/4$ and $b(\gamma,n)=(\gamma-2\log n)/(8m^*+4)$. 
Since $\gamma\ge 64(2m^\ast+1)+2\log(2n)$, we obtain 
$b(\gamma,2n)\ge 16$ and $\sqrt{2(k-1)(k+2)\gamma}\le (k-1)\gamma/4$.

Similar to the proof of Lemma \ref{lemma.addk.A},  by Bonferroni correction, we have the probability $\p\left(\min_{k\ge 2,\tau_{1:k}}\left\{\Loss({\bS}; \tau_{1:k})+(k-1)\gamma\right\}- \Loss^{\ast}(\bS)\le a(\gamma, 2n)\right)$ is upper bounded by 
$
\sum_{k=2}^\infty {2n \choose k} \p\left( \chi^2_{k+2}\ge (k-1)\gamma-a(\gamma, 2n)\right)
$

Applying again  (\ref{ineq.chisq.upper}) in Lemma \ref{lemma.chi-square},
\begin{align*}
&{2n \choose k} \p\left( \chi^2_{k+2}\ge (k-1)\gamma-a(\gamma, 2n)\right)\\
<& \dfrac{(2n)^k}{k!} \exp\left(-\dfrac{(k-1)\gamma-a(\gamma, 2n)-\sqrt{2(k+2)(k-1)\gamma}}{2}\right)\\
\le& \dfrac{(2n)^k}{k!} \exp\left(-\dfrac{(k-1)\gamma}{4}\right)\\
\le& \dfrac{1}{k!} \exp\left(-\dfrac{(k-1)\gamma-4k\log(2n)}{4}\right)\\
\le& \dfrac{1}{2^{k-1}} \exp\left(-\dfrac{(k-1)(\gamma-8\log(2n))}{4}\right).
\end{align*}
Therefore
\begin{align*}
&\p\left(\min_{k\ge 2,\tau_{1:k}}\left\{\Loss({\bS}; \tau_{1:k})+(k-1)\gamma\right\}- \Loss^{\ast}(\bS)\le a(\gamma, 2n)\right)\\
<&\sum_{k=2}^{\infty}\dfrac{1}{2^{k-1}} \exp\left(-\dfrac{(k-1)(\gamma-8\log(2n))}{4}\right)=\exp\left(-\dfrac{\gamma-8\log(2n)}{4}\right).
\end{align*}

Moreover, since  $b(\gamma, n)\ge 16$, which leads to $2\sqrt{b(\gamma,n)}\le b(\gamma,n)/2$, therefore
\begin{align*}
\p\left(\Loss^\ast({\bS})-\min_{\tau_1}\Loss(\bS; \tau_1)\ge b(\gamma, 2n)\right)
\le &(2n-1)\p\bigg(\Loss^\ast(\bS)-\Loss(\bS; \tau_1)\ge b(\gamma, 2n)\bigg)\\
=&(2n-1)\p\left(\chi^2_2\ge b(\gamma, 2n)\right)\\
< &2n \exp\left(-\dfrac{b(\gamma,2n)- 2\sqrt{b(\gamma, 2n)}
}{2}\right)\nonumber\\
\le&\exp\left(-\dfrac{b(\gamma,2n)-4\log(2n)}{4}\right)\nonumber\\
 = &\exp\left(-\dfrac{\gamma-(8m^\ast+6)\log (2n)}{16(2m^\ast+1)}\right),\nonumber
\end{align*}
which completes the proof.

\end{proof}

\medskip

\begin{lemma}\label{lemma.drop1.A}  Consider any $t$ and model such that $\bS=x_{t+1:t+2n}$ is a region then contains a single changepoint  at $\tau^\ast=t+n$ and $\Delta$ is the absolute difference between the true means before and after the change. For any $5\le z\le {n\Delta^2}/8$ we have
$$\p\left(\Loss\left({\bS};\varnothing\right)-\Loss^\ast\left({\bS}\right)\le z\right)\le 2\exp\left(-\dfrac{z}{20}\right).
$$
\end{lemma}
\begin{proof} 
It is straightforward to show that $\Loss\left({\bS};\varnothing\right)-\Loss\left({\bS, \tau^\ast}\right)\sim \chi^2_{1}(\nu)$ with non-centrality parameter 
$\nu=n\Delta^2/2$, and $\Loss^\ast\left({\bS}\right)-\Loss\left({\bS, \tau^\ast}\right)\sim \chi^2_2$.
Therefore as long as $5\le z\le n\Delta^2/8$,
\begin{align}
&\p\left(\Loss\left({\bS};\varnothing\right)-\Loss^\ast\left({\bS}\right)\le z\right)\\
\le &\p\left(\Loss\left({\bS};\varnothing\right)-\Loss^\ast\left({\bS}\right)\le z + \Loss^\ast\left({\bS}\right)-\Loss\left({\bS, \tau^\ast}\right)\right)\nonumber\\
\le& 
\p\left(\Loss\left({\bS};\varnothing\right)-\Loss\left({\bS};\tau^\ast\right)\le 2z\right)+\p\left(\Loss^\ast\left({\bS}\right)-\Loss\left({\bS, \tau^\ast}\right)\ge z\right)\nonumber
\\
\le&\p\left(\chi^2_1(\nu)\le 2z\right)+\p\left(\chi^2_2\ge z\right)\nonumber\\
\le&\exp\left(-\dfrac{(1+\nu-2z)^2}{4+8\nu}\right)+\exp\left(-\dfrac{z-2\sqrt{z}}{2}\right)\\
<&2\exp\left(-\dfrac{z}{20}\right),\nonumber
\end{align}
where the second last inequality follows from Lemma \ref{lemma.chi-square}.

\end{proof}
\medskip

As  Lemmas \ref{lemma.addk.A}, \ref{lemma.addk2.A} and \ref{lemma.drop1.A} hold, respectively, for any $\bS\in \{\bS_{1,n}(t)\}$, $\{\bS_{2,n}(t)\}$ and $\{\bS_{\Delta,n}(t)\}$, it is straightforward to obtain Propositions \ref{prop:mean1} and \ref{prop:mean2}.

\subsection{Proof of Theorem \ref{thm:mean}}


We take $\beta=(2+\epsilon)\log T$, with a suitable choice of $n_{1:m^\ast}$ such that  $$n_j=\min\left\{\dfrac{8(\beta+a(\beta, T))}{\Delta_j^2}, \delta_j, \delta_{j+1}\right\},$$ where $a(\beta, T)=(\beta-2\log T)/4$ as indicates in Proposition \ref{prop:mean2}. Let $b(\beta, T)=(\beta-2\log T)/4(2m^\ast+1)$. 

First, we show that the choice of $\beta$ satisfy the requirements in Theorem \ref{thm.global}. 
As we require $\delta_T\Delta_T^2\ge (16+10\epsilon)\log T$, it follows that
\begin{align}\label{eq.nj.mean}
n_j= \dfrac{8(\beta+a(\beta, T))}{\Delta_j^2}=\dfrac{(16+10\epsilon)\log T}{\Delta_j^2}.
\end{align}

If $T\ge \max\left\{\exp\left({64}/{\epsilon}+{128}/{\epsilon^2}\right),\exp\big({(128m^\ast+64)}/{\epsilon}\big)\right\}$, we have $\beta\ge \gamma^{(1)}_{T}$, where
\begin{equation*}
\gamma_T^{(1)}=\max\left\{(2+\epsilon)\log T, 2\log n+8\sqrt{16+2\log T}+32, 2\log T + 32(2m^\ast + 1)\right\},
\end{equation*} 
as defined in Proposition \ref{prop:mean2}.

Moreover, note that if $T$ is large enough and
\begin{align*}
\Delta_T^2\ge \dfrac{2(16+10\epsilon)\log T}{T^{1/(4m^\ast+3)}},
\end{align*}
we have $\epsilon\log T \ge 64(2m^\ast+1)$  and $2\log T\ge (8m^\ast+6)\log\left(2 n_j\right)$ for all $j=1, 2,\dots,m^\ast$. Therefore $\beta\ge \max_{j}\gamma^{(2)}_{ n_j}$, 
where each of $\gamma^{(2)}_{ n_j}$ has the following form 
\begin{equation*}
 \gamma_{n_j}^{(2)}=\max\left\{(8m^\ast+6+\epsilon)\log (2n_j), 2\log (2n_j)+64(2m^\ast+1)\right\},
 \end{equation*}
as defined in Proposition \ref{prop:mean2}.


Altogether, as long as $T$ is large enough, we have $\beta\ge \max\bigg\{\gamma_{T}^{(1)}, \max_j\gamma^{(2)}_{n_j}\bigg\}$ and $a(\beta, T)>2m^\ast b(\beta,T)$.

Next, we give the probability bound for local region conditions. By Proposition \ref{prop:mean1}, we can derive the formula of $p_1(\gamma, n), p_2(\gamma, n), p_3(\gamma, n)$ and $p_4(\gamma, n)$.
Since $\beta=(2+\epsilon)\log T$,  it is straightforward to verify that 
 $$\min\bigg\{p_1(\beta, T)/2, p_2(\beta, T), p_3\left(\beta, \max_j n_j\right), p_4\left(\beta, \max_j n_j\right)\bigg\} \le T^{-\epsilon/16(2m^\ast+1)}.$$

In addition, note that $\bar \mS=\min_j\Delta_j^2n_j/2$, using equation (\ref{eq.nj.mean}), we have
$\bar \mS\ge 20$ as $T$ is large enough. Combined with Proposition \ref{prop:mean2} we obtain
\begin{align*}
p_5\left(\bar\mS, \beta+a\left(\beta,T\right)\right)=&2\exp\left(-\dfrac{\beta+a(\beta, T)}{20}\right)
= 2T^{-\frac{2+5/4\epsilon}{20}} \le 2T^{-\frac{\epsilon}{32m^\ast+16}}.
\end{align*}

Hence, following  Theorem \ref{thm.global}, we have
\begin{align*}
&\p\left(\hat{m}=m, |\htau_j-\tau^\ast_j|\le n_j \mbox{ for all }j=1,\dots, m^\ast \right)\\
\ge& 1-(m^\ast+1)p_1\left(\beta,T\right)-(m^\ast+1) p_2(\beta, T)-m^\ast p_3\left(\beta, \max_j n_j\right)\\
&-m^\ast p_4\left(\beta, \max_j n_j\right)-m^\ast p_5\left(\sigst,\beta+a(\beta,T)\right)\\
\ge& 1-(7m^\ast+3)T^{-\epsilon/(32m^\ast+16)}.
\end{align*}

\bigskip

\section{Proofs of Section~\ref{sec:change-in-slope}} \label{suppsec:change-in-slope}

\subsection{Proof for Proposition \ref{prop:slope1} and \ref{prop:slope2}}
\begin{lemma}\label{lemma.addk.B} For any $t$ and any model, if $\bS=x_{t+1:t+n}$ is a region that contains no true changepoint, then for any $\gamma\ge \gamma_{n}^{(1)}$, where
\begin{equation}
 \gamma_{n}^{(1)}=\max\left\{(2+\epsilon)\log n, 2\log n+4\sqrt{9+3\log n}+12, 2\log n+96(2m^\ast+1)\right\},
\end{equation}
 where  $\epsilon>0$, we have 
 $$\p\left(\min_{1\le k,\tau_{1:k}}\{\Loss\left(\bS;\tau_{1:k}\right)+k\gamma\}-\Loss^\ast\left(\bS\right)\le \dfrac{\gamma-2\log n}{6}\right)< 2\exp\left(-\dfrac{\gamma-2\log n}{6}\right),$$
  and 
  $$
\p\bigg( \Loss^\ast(\bS)- \Loss(\bS;\varnothing) \geq \dfrac{\gamma-2\log n}{6(2m^\ast+1)} \bigg) \le \exp\left(-\dfrac{\gamma-2\log n}{24(2m^\ast+1)}\right).
$$ 
\end{lemma}

\begin{proof}

Note that by Lemma \ref{lem.chi-square2.slope}, for any $k$ and $\tau_{1:k}$,  we  have
$\Loss^\ast(\bS)-\Loss(\bS; \tau_{1:k})\sim\chi^2_{k+2}$ and
$\Loss^\ast\left(\bS\right)-\Loss({\bS};\varnothing)\sim \chi^2_{2}$.

Therefore, similar to the proof of Lemma \ref{lemma.addk.A}, let $a(\gamma,n)=\dfrac{\gamma-2\log n}{6}$,
\begin{align*}
&\p\left(\min_{k,\tau_{1:k}}\{\Loss(\bS; \tau_{1:k})+k\gamma-\Loss^{\ast}(\bS)\}\le a(\gamma,n) \right)\\
&\le \sum_{k=1}^n \p\left(\max_{\tau_{1:k}}\left[\Loss^{\ast}({\bS})-\Loss({\bS};\tau_{1:k})\right]\ge k\gamma-a(\gamma,n)\right)\\
&\le \sum_{k=1}^\infty {n \choose k} \p\left( \chi^2_{k+2}\ge k\gamma-a(\gamma,n)\right)
\end{align*}

Applying (\ref{ineq.chisq.upper}) in Lemma \ref{lemma.chi-square}, we have
\begin{align*}
{n \choose k} \p\left( \chi^2_{k+2}\ge k\gamma-a(\gamma,n)\right)& \le\dfrac{n^k}{k!}  \exp\left(-\dfrac{k\gamma-(\gamma-2\log n)/6-\sqrt{2k(k+2)\gamma}}{2}\right)
\\&\le\dfrac{1}{k!}  \exp\left(-\dfrac{(k-1/6)(\gamma-2\log n)-\sqrt{2k(k+2)\gamma}}{2}\right)
\end{align*}

Note that as long as $\gamma\ge 2\log n+4\sqrt{9+3\log n}+12$, we have 
\begin{align}
\sqrt{2k(k+2)\gamma}\le \dfrac{k(\gamma-2\log n)}{2},
\end{align}

which leads to 
\begin{align*}
\p\left(\min_{k,\tau_{1:k}}\{\Loss(\bS; \tau_{1:k})+k\gamma-\Loss^{\ast}(\bS)\}\le a(\gamma,n) \right)&\le\sum_{k=1}^\infty\dfrac{1}{k!}\exp\left(\-\dfrac{(k/2-1/6)(\gamma-2\log n)}{2}\right)\\
&\le \sum_{k=1}^\infty\dfrac{1}{2^{k-1}} \exp\left(-\dfrac{\gamma-2\log n}{6}\right)\\
&=2\exp\left(-\dfrac{\gamma-2\log n}{6}\right).
\end{align*}

Let $b(\gamma,n)=(\gamma-2\log n)/(12m^*+6)$. 
If $\gamma\ge 2\log n+96(2m^\ast+1)$, then $b(\gamma,n)\ge 16$, and, as a result, $\sqrt{4b(\gamma,n)}\le b(\gamma,n)/2 $. Hence
\begin{align*}
\p\bigg( \Loss^\ast(\bS)- \Loss(\bS;\varnothing) \geq b(\gamma,n) \bigg)=&\p\left(\chi^2_2\ge  b(\gamma,n)\right)\\
\le&\exp\left(-\dfrac{b(\gamma,n)-\sqrt{4b(\gamma,n)}}{2}\right)\\
\le&\exp\left(-\dfrac{b(\gamma,n)}{4}\right)
\\= &\exp\left(-\dfrac{\gamma-2\log n}{24(2m^\ast+1)}\right).
\end{align*}
\end{proof}

\medskip

\begin{lemma}\label{lemma.addk2.B}For any $t$, any $n\geq 4$ and any model, if $\bS=x_{t+1:t+2n}$ is a region that contains a single changepoint  at $\tau^\ast=t+n$, then for any $\gamma\ge \gamma_n^{(2)}$ where 
\begin{equation*}
\gamma_n^{(2)}=\max\left\{(3+\epsilon)\log (2n), 2\log(2n)+32\log(C\log(2n)),   2\log(2n)+972(2m^\ast+1), 3240\right\}
\end{equation*}
where $\epsilon>0$ and $C$ is a positive constant not depend on $n$, we have  
 $$\p\left(\min_{k\ge 2,\tau_{1:k}}\{\Loss({\bS}; \tau_{1:k})+(k-1)\gamma\}- \Loss^{\ast
}(\bS)\le \dfrac{\gamma-2\log (2n)}{6}\right)< \dfrac{9}{4} \exp\left(-\dfrac{\gamma-3\log(2n)}{3}\right),$$
and 
$$
\p\left(\Loss^\ast({\bS})-\min_{\tau_1}\Loss(\bS; \tau_1)\le \dfrac{\gamma-2\log (2n)}{6(2m^\ast+1)}\right) \le \exp\left(-\dfrac{\gamma-2\log (2n)}{24(2m^\ast+1)}\right).
$$
\end{lemma}

\begin{proof}
By Lemma \ref{lem.chi-square.slope}, note that for any $\tau_{1:k}$ on $\bS$ and $k\ge 1$,
 $$
\Loss^\ast(\bS)-\Loss({\bS};\tau_{1:k})\le\Loss^\ast({\bS})-\Loss({\bS};\tau_{1:k},\tau^\ast)\sim\chi^2_{k+3}, \mbox{ and }
$$
$$
\Loss^\ast(\bS)-\Loss\left(\bS;\tau^\ast\right)\sim \chi^2_3.
$$

Define $a(\gamma,2n)=(\gamma-2\log(2n))/6$.
Since $\gamma>\max\{(3+\epsilon)\log(2n), 3240\}$, we have 
\begin{align}\label{ineq.slope2}
a(\gamma, 2n)\le \dfrac{\gamma}{6} \quad \mbox{ and }\quad \sqrt{2(k-1)(k+3)\gamma}\le \dfrac{(k-1)\gamma}{18}.
\end{align}

Therefore, similar to the proof of Lemma \ref{lemma.addk2.A}, for $k\ge 4$ we have
\begin{align*}
&\p\left(\min_{k\ge 4,\tau_{1:k}}\left\{\Loss({\bS}; \tau_{1:k})+(k-1)\gamma\right\}- \Loss^{\ast}(\bS)\le a(\gamma, 2n)\right)\\
\le& \sum_{k=4}^{\infty}\p\left(\max_{\tau_{1:k}}\left(\Loss(\bS)^{\ast}-\Loss(\bS; \tau_{1:k},\tau^\ast)\right)\ge (k-1)\gamma-a(\gamma, 2n)\right)\\
<& \sum_{k= 4}^{\infty}{2n \choose k} \exp\left(-\dfrac{(k-1)\gamma-a(\gamma, 2n)-\sqrt{2(k-1)(k+3)\gamma}}{2}\right)
\end{align*}
where the last inequality is due to Bonferroni correction and (\ref{ineq.chisq.upper}) in Lemma \ref{lemma.chi-square}. Together with (\ref{ineq.slope2}), we have
\begin{align*}
&\p\left(\min_{k\ge 4,\tau_{1:k}}\left\{\Loss({\bS}; \tau_{1:k})+(k-1)\gamma\right\}- \Loss^{\ast}(\bS)\le a(\gamma, 2n)\right)\\\le& \sum_{k= 4}^{\infty}{2n \choose k} \exp\left(-\dfrac{(k-1)\gamma-\gamma/6-(k-1)\gamma/18}{2}\right)\\
\le& \sum_{k=4}^{\infty}\dfrac{1}{k!}\exp\left(-\dfrac{[18(k-1)/17-1/6]\gamma-2k\log(2n)}{2} \right)\\
<& \sum_{k=4}^{\infty}\dfrac{1}{k!}\exp\left(-\dfrac{4\left(\gamma-3\log(2n)\right)}{3}\right)\\
<& \dfrac{1}{4}\exp\left(-\dfrac{4\left(\gamma-3\log(2n)\right)}{3}\right)\le \dfrac{1}{4}\exp\left(-\dfrac{\gamma-3\log(2n)}{3}\right)
\end{align*}

Now we only need to handle the case when $k=2$ and $3$, note that
$$
\max_{\tau_1,\tau_2}\left\{\Loss^\ast(\bS)-\Loss_n(\bS; \tau_{1},\tau_2)\right\}\le\max_{\tau_1,\tau_2}\left\{\Loss^\ast(\bS)-\Loss_n(\bS; \tau_{1:2}, \tau^\ast)\right\},
$$
and
$$
\max_{\tau_{1:3}}\left\{\Loss^\ast(\bS)-\Loss_n(\bS; \tau_{1:3})\right\}\le\max_{\tau_{1:3}}\left\{\Loss^\ast(\bS)-\Loss_n(\bS; \tau_{1:3},\tau^\ast)\right\}.
$$ Using the results from Lemmas \ref{lemma.max2.B} and \ref{lemma.max3.B}, if $\gamma>\max\bigg\{240, 24\log(C\log(2n))\bigg\}$, where $C=\max\big\{C_1''', C_2''', C_3''''\big\}$ is a positive constant and $C_1''', C_2''',C_3'''$ are constants introduced in Section \ref{sec:D.3}, we have both the events
 $$\max_{\tau_1,\tau_2}\left\{\Loss^\ast(\bS)-\Loss_n(\bS; \tau_{1:2}, \tau^\ast)\right\}\ge \gamma-a(\gamma, 2n)$$
 and 
$$ \max_{\tau_1,\tau_2,\tau_3}\left\{\Loss^\ast(\bS)-\Loss_n(\bS; \tau_{1:3}, \tau^\ast)\right\}\ge2\gamma-a(\gamma, 2n)$$
hold with probability less than $\exp\left(-\dfrac{\gamma-3\log(2n)}{3}\right)$. 
Therefore, by the union bound,we have
\begin{align*}
\p\left(\min_{k\ge 2,\tau_{1:k}}\left\{\Loss(\bS; \tau_{1:k})+(k-1)\gamma\right\}-\Loss^\ast(\bS)\le \dfrac{\gamma-2\log(2n)}{6} \right)< \dfrac{9}{4}\exp\left(-\dfrac{\gamma-3\log(2n)}{3}\right).
\end{align*}

Let $b(\gamma,n)=(\gamma-2\log n)/(12m^*+6)$, by Lemma \ref{lemma.max1.B}, if  $$\gamma>\max\bigg\{2\log(2n)+32\log(C\log (2n)), 2\log(2n)+972(2m^\ast+1)\bigg\},$$ we have
\begin{align*}
\p\left(\Loss^\ast({\bS})-\min_{\tau_1}\Loss(\bS; \tau_1)\ge b(\gamma, 2n)\right)
\le &\p\bigg(\Loss^\ast(\bS)-\min_{\tau_1}\Loss(\bS;\tau_1,\tau^\ast)\ge b(\gamma, 2n)\bigg)\\
 \le &\exp\left(-\dfrac{\gamma-2\log (2n)}{24(2m^\ast+1)}\right),\nonumber
\end{align*}
which completes the proof.

\end{proof}

\medskip

\begin{lemma}\label{lemma.drop1.B} For any $t$ and any model, if $\bS=x_{t+1:t+2n}$ is a region that contains a single changepoint at $\tau^\ast=t+n$ and $\Delta$ is the absolute difference between the true slopes before and after the change, then for any $8\le z\le {n^3\Delta^2}/100$ and $n\ge 2$, we have
$$\p\left\{\Loss\left({\bS};\varnothing\right)-\Loss^\ast\left({\bS}\right)\le z\right\}\le \exp\left(-{z}/{18}\right).
$$
\end{lemma}
\begin{proof} 

Lemma \ref{lem.chi-square.slope} suggests that $\Loss(\bS,\varnothing)-\Loss(\bS,\tau^\ast)$ and $\Loss^\ast(\bS)-\Loss(\bS,\tau^\ast)$ follows $\chi^2_1(\nu)$ and $\chi^2_3$, respectively, where
$$\nu=\Delta_n^2\dfrac{n(n+1)(n-1)}{24}\dfrac{4n^2+2}{4n^2-1}\ge \dfrac{\Delta_n^2n^3}{25}.$$ 
Following the same argument as in the proof of Lemma \ref{lemma.drop1.A}, as long as $8\le z\le n^3\Delta^2/100$,
\begin{align}
&\p\bigg(\Loss\left({\bS};\varnothing\right)-\Loss^\ast\left({\bS}\right)\le z + \Loss^\ast\left({\bS}\right)-\Loss\left({\bS, \tau^\ast}\right)\bigg)\nonumber\\
\le& 
\p\bigg(\Loss\left({\bS};\varnothing\right)-\Loss\left({\bS};\tau^\ast\right)\le 2z\bigg)+\p\bigg(\Loss^\ast\left({\bS}\right)-\Loss\left({\bS, \tau^\ast}\right)\ge z\bigg)\nonumber
\\
\le&\p\left(\chi^2_1(\nu)\le 2z\right)+\p\left(\chi^2_3\ge z\right)\nonumber\\
\le&\exp\left(-\dfrac{(1+\nu-2z)^2}{4+8\nu}\right)+\exp\left(-\dfrac{z-\sqrt{6z}}{2}\right)\\
\le&2\exp\left(-\dfrac{z}{20}\right),\nonumber
\end{align}
where the second inequality follows from Lemma \ref{lemma.chi-square}.
\end{proof}

 Note that  Lemmas \ref{lemma.addk.B}, \ref{lemma.addk2.B} and \ref{lemma.drop1.B} hold for any $\bS\in \{\bS_{1,n}(t)\}$, $\{\bS_{2,n}(t)\}$ and $\{\bS_{\Delta,n}(t)\}$, respectively. Therefore, it is straightforward to obtain Propositions \ref{prop:slope1} and \ref{prop:slope2}.

 \subsection{Proof of Theorem \ref{thm:slope}}
Similar to the proof of Theorem \ref{thm:mean}, we take $\beta=(2+\epsilon)\log T$ and  $$n_j=\min\left\{\left(\dfrac{100(\beta+a(\beta, T))}{\Delta_j^2}\right)^{1/3},\, \delta_j,\, \delta_{j+1}\right\},$$ where $a(\beta, T)=(\beta-2\log T)/6$ as indicates in Proposition \ref{prop:slope1} and let $b(\beta, T)=(\beta-2\log T)/6(2m^\ast+1)$. 

Therefore, since $\delta_T^3\Delta_T^2\ge (200+350\epsilon/3)\log T=100(\beta+a(\beta, T)$, we have
$$n_j= \left(\dfrac{(200+350\epsilon/3)\log T}{\Delta_j^{2}}\right)^{1/3}.$$  Combined with the assumption that $\Delta_T^2\ge \dfrac{8(200+250\epsilon/3)\log T}{T^2}$, thus $\max_j n_j\le T^{2/3}/2$, which leads to $\log\left(2\max_j n_j\right)\le 2/3\log T$.

Therefore, it is straightforward to verify that as long as $T$ is large enough, we have $\beta=(2+\epsilon)\log T\ge \max\{\gamma^{(1)}_{T}, \max_{j}\gamma^{(2)}_{n_j}\}$, 
where
\begin{align*} 
 \gamma_{T}^{(1)}&=\max\left\{(2+\epsilon)\log T, 2\log t+4\sqrt{9+3\log t}+12, 2\log T+96(2m^\ast+1)\right\},\\
\gamma_{n_j}^{(2)}&=\max\left\{(3+\epsilon)\log (2n_j), 2\log(2n_j)+32\log(C\log(2n_j)),   2\log(2n_j)+972(2m^\ast+1), 3240\right\},
\end{align*}
as defined in Proposition \ref{prop:slope1}. Therefore we can verify that
 
 $$\min\bigg\{p_1(\beta, T)/2, p_2(\beta, T), 4p_3\left(\beta, \max_j n_j\right)/9, p_4\left(\beta, \max_j n_j\right)\bigg\} \le T^{-\epsilon/24(2m^\ast+1)},$$
 where  $p_1(\gamma, n), p_2(\gamma, n), p_3(\gamma, n), p_4(\gamma, n)$ can be obtained from Lemmas \ref{lemma.addk.B}, \ref{lemma.addk2.B} and \ref{lemma.drop1.B}.

Applying Proposition \ref{prop:slope2},
 note that if $\bar \mS\ge 4\left(\beta+a(\beta, T)\right)\ge 32$, which is true as $T$ is large enough,
 we have 
 $$p_5\left(\bar\mS, \beta+a\left(\beta,T\right)\right)\le 2T^{-(2+7\epsilon/6)/20} \le 2T^{-\epsilon/(48m^\ast+24)}.$$

 Therefore, 
 \begin{align*}
 &\p\bigg(\hat{m}=m, \max_j|\htau_j-\tau^\ast_j|^3\Delta_j^2\le (200+350\epsilon/3)\log T \bigg)\\ 
 \ge&1-(m^\ast+1)p_1\left(\beta,T\right)-(m^\ast+1) p_2(\beta, T)-m^\ast p_3\left(\beta, \max_j n_j\right)\\
 &-m^\ast p_4\left(\beta, \max_j n_j\right)-m^\ast p_5\left(\sigst,\beta+a(\beta,T)\right)\\
 \ge& 1-(33m^\ast/4+3)T^{-\epsilon/(48m^\ast+24)}.
 \end{align*}
 \qed

\medskip

\section{Proofs of in Section~\ref{sec:change-in-spike}} \label{suppsec:change-in-spike}

\subsection{Proofs for Propositions \ref{prop:spike1} and \ref{prop:spike2}} 

\begin{lemma}\label{lemma.addk.C}For any $t$ and any model, if $\bS=x_{t+1:t+n}$ is a region contains no true changepoint, for any $$\gamma\ge \max\left\{(2+\epsilon)\log n, 2\log n+8\sqrt{16+2\log n}+32, 2\log n + 32(2m^\ast + 1)\right\},$$ where $\epsilon$ is an arbitrarily small positive constant, we have 
 $$\p\left(\min_{1\le k,\tau_{1:k}}\{\Loss\left(\bS;\tau_{1:k}\right)+k\gamma\}-\Loss^\ast\left(\bS\right)\le \dfrac{\gamma-2\log n}{4}\right)< 2\exp\left(-\dfrac{\gamma-2\log n}{4}\right),$$
  and 
  $$
\p\bigg( \Loss^\ast(\bS)- \Loss(\bS;\varnothing) \geq \dfrac{\gamma-2\log n}{4(2m^\ast+1)} \bigg) \le \exp\left(-\dfrac{\gamma-2\log n}{16(2m^\ast+1)}\right).
$$
\end{lemma}

\begin{proof}
We only need to consider the case that $t+1\le \tau_{1:k}\le t+n$. Without loss of generality, let $t=0$, i.e., consider $\bS=x_{1:n}$, where $x_s=\theta^\ast\alpha^{s-1}+\eps_s$.

Note that $\Loss^\ast(\bS)=\sum_{s=1}^{n}{\eps_s^2}/{\sigma^2}$.

For any $1\le \tau_{1:k}<n$, let $\tau_0=0$ and $\tau_{k+1}=n$, it is easy to derive
\begin{align*}
\Loss(\bS; \tau_{1:k})=\sum_{j=1}^{k+1}\sum_{s=\tau_{j-1}+1}^{\tau_j}\dfrac{1}{\sigma^2}\left(x_{s}-\hat\theta_{j}\alpha^{s-\tau_{j-1}-1}\right)^2,
\end{align*}
where $$\hat\theta_j=\dfrac{\sum_{s={\tau_{j-1}+1}}^{\tau_j}x_s\alpha^{s-\tau_{j-1}-1}}{\sum_{s={\tau_{j-1}+1}}^{\tau_j}\alpha^{2(s-\tau_{j-1}-1)}}=\theta^\ast\alpha^{\tau_{j-1}}+\dfrac{\sum_{s={\tau_{j-1}+1}}^{\tau_j}\eps_s \alpha^{s-\tau_{j-1}-1}}{\sum_{s={\tau_{j-1}+1}}^{\tau_j}\alpha^{2(s-\tau_{j-1}-1)}}.$$

Therefore for each $j=1,\dots, m+1$,
\begin{align*}
&\sum_{s=\tau_{j-1}+1}^{\tau_j}\dfrac{1}{\sigma^2}\eps^2_s-\sum_{s=\tau_{j-1}+1}^{\tau_j}\dfrac{1}{\sigma^2}\left(x_{s}-\hat\theta_{j+1}\alpha^{s-\tau_{j-1}-1}\right)^2\\
=&\sum_{s=\tau_{j-1}+1}^{\tau_j}\dfrac{1}{\sigma^2}\eps^2_s-\sum_{s=\tau_{j-1}+1}^{\tau_j}\dfrac{1}{\sigma^2}\left(\eps_s-\alpha^{s-\tau_{j-1}-1}\dfrac{\sum_{s={\tau_{j-1}+1}}^{\tau_j}\eps_s \alpha^{s-\tau_{j-1}-1}}{\sum_{s={\tau_{j-1}+1}}^{\tau_j}\alpha^{2(s-\tau_{j-1}-1)}}\right)^2\\
=&\dfrac{\left(\sum_{s={\tau_{j-1}+1}}^{\tau_j}\eps_s \alpha^{s-\tau_{j-1}-1}\right)^2}{{\sigma^2}\sum_{s={\tau_{j-1}+1}}^{\tau_j}\alpha^{2(s-\tau_{j-1}-1)}}.
\end{align*}
Note that $\dfrac{\sum_{s={\tau_{j-1}+1}}^{\tau_j}\eps_s \alpha^{t-\tau_{j-1}-1}}{\sigma\left(\sum_{s={\tau_{j-1}+1}}^{\tau_j}\alpha^{2(s-\tau_{j-1}-1)}\right)^{1/2}}\sim N(0,1)$. Since $\{\eps_s\}_{s=1}^{T}$ are i.i.d, we have
$
\Loss^\ast\left(\bS\right)-\Loss({\bS};\tau_{1:k})\sim \chi^2_{k+1}.
$ 

Similarly, we have 
$$
\Loss(\bS, \varnothing)=\sum_{s=1}^n\dfrac{1}{\sigma^2}\left(\eps_s-\alpha^{s-1}\dfrac{\sum_{s=1}^n\eps_s\alpha^{(s-1)}}{\sum_{s=1}^n\alpha^{2(s-1)}}\right)^2=\sum_{s=1}^n\dfrac{\eps_s^2}{\sigma^2}-\dfrac{\left(\sum_{s=1}^n\eps_s\alpha^{(s-1)}\right)^2}{\sigma^2\sum_{s=1}^n\alpha^{2(s-1)}},
$$
which leads to
$\Loss^\ast\left(\bS\right)-\Loss({\bS};\varnothing)\sim \chi^2_{1}$.

Using the same argument in the proof of Lemma \ref{lemma.addk.A} completes the proof.
\end{proof}

\begin{lemma}\label{lemma.addk2.C} For any $t$ and any model, if $\bS=x_{t+1:t+2n}$ is a region contains a single changepoint  at $\tau^\ast=t+n$, for any $$\gamma\ge \max\{(8m^\ast+6+\epsilon)\log(2n),  2\log(2n)+64(2m^\ast+1)\},$$  where $\epsilon$ is an arbitrarily small positive constant, for $n\ge 4$, we have  
 $$\p\left(\min_{k\ge 2,\tau_{1:k}}\{\Loss({\bS}; \tau_{1:k})+(k-1)\gamma\}- \Loss^{\ast
}(\bS)\le \dfrac{\gamma-2\log (2n)}{4}\right)<  \exp\left(-\dfrac{\gamma-8\log n}{4}\right),$$
and 
$$
\p\left(\Loss^\ast({\bS})-\min_{\tau_1}\Loss(\bS; \tau_1)\le \dfrac{\gamma-2\log (2n)}{4(2m^\ast+1)}\right) \le \exp\left(-\dfrac{\gamma-(8m^\ast+6)\log (2n)}{16(2m^\ast+1)}\right).
$$
\end{lemma}

\begin{proof} 
For any $\tau_{1:k}$, note that
\begin{align} \lbl{eq:spike1}
\Loss^\ast(\bS)-\Loss({\bS};\tau_{1:k})\le&\Loss^\ast({\bS})-\Loss({\bS};\tau_{1:k},\tau^\ast)\nonumber\\
=&\Loss^\ast\left(x_{(t+1):(t+n)}\right)+\Loss^\ast\left(x_{(t+n+1):(t+2n)}\right)\nonumber\\
&-\Loss\left(x_{(t+1):(t+n)}; \tau_{1:k}\right)-\Loss\left(x_{(t+n+1):(t+2n)}; \tau_{1:k}\right)
\end{align}
From the proof of Lemma \ref{lemma.addk.C}, we have (\ref{eq:spike1}) follows a chi-square distribution with degrees of freedom $k+2$. Similarly, we have
\begin{align}
\Loss^\ast(\bS)-\min_{\tau_1}\Loss\left(\bS;\tau_1\right)\le \Loss^\ast(\bS)-\Loss(\bS; \tau^\ast)\sim \chi^2_2.
\end{align}
Using the same argument as in the proof of Lemma \ref{lemma.addk2.A}, we obtain the results.
\end{proof}

\begin{lemma}\label{lemma.drop1.C}If $\bS=x_{t+1:t+2n}$ is a region contains a single changepoint  at $\tau^\ast=t+n$ and $\Delta$ be the absolute difference between the true means before and after the change. For any $5\le z\le \dfrac{\Delta^2}{(1-\alpha^{2n})(1-\alpha^2)}$, we have
$$\p\left\{\Loss\left({\bS};\varnothing\right)-\Loss^\ast\left({\bS}\right)\le z\right\}\le \exp\left(-{z}/{20}\right).
$$
\end{lemma}
\begin{proof} Without loss of generality we let $t=0$, i.e., $\bS=x_{1;2n}$ with $\tau^\ast=n$ is a changepoint. Therefore, we write that
$$
\begin{cases}
x_s=\theta_1^\ast\alpha^{s-1}+\eps_s,\quad 1\le s\le n\\
x_s=\theta_2^\ast\alpha^{s-n-1}+\eps_s, \quad n+1\le s\le 2n.
\end{cases}
$$
First, we have $\Loss^\ast(\bS)=\sum_{s=1}^{2n}{\eps_s^2/\sigma^2}$ and  
\begin{align} \lbl{eq:spike.omit1}
\Loss(\bS; \tau^\ast)=\sum_{s=1}^{n}\dfrac{1}{\sigma^2}\left(x_{s}-\hat\theta_{1}\alpha^{s-1}\right)^2+\sum_{s=n+1}^{2n}\dfrac{1}{\sigma^2}\left(x_{s}-\hat\theta_{2}\alpha^{s-n-1}\right)^2,
\end{align}
where $$\hat\theta_1=\theta_1^{\ast} +  \dfrac{\sum_{s={1}}^{n}\eps_s \alpha^{s-1}}{\sum_{s=1}^{n}\alpha^{2(s-1)}} \quad, \mbox{ and }\quad  \hat\theta_2=\theta_2^{\ast} + \dfrac{\sum_{s={n+1}}^{2n}\eps_s \alpha^{s-n-1}}{\sum_{s=n+1}^{2n}\alpha^{2(s-n-1)}}.
$$
Using the similar argument as in the proof of Lemma \ref{lemma.addk.C},
we can rewrite (\ref{eq:spike.omit1}) as 
\begin{align*}
\Loss(\bS; \tau^\ast)= &\sum_{s=1}^{n}\dfrac{1}{\sigma^2}\left(\eps_s-\alpha^{s-1}\dfrac{\sum_{s=1}^n\eps_s\alpha^{(s-1)}}{\sum_{s=1}^n\alpha^{2(s-1)}}\right)^2\\
&+\sum_{s=n+1}^{2n}\dfrac{1}{\sigma^2}\left(\eps_s-\alpha^{s-n-1}\dfrac{\sum_{s=n+1}^{2n}\eps_s\alpha^{(s-n-1)}}{\sum_{s=n+1}^{2n}\alpha^{2(s-n-1)}}\right)^2\\
=&\sum_{s=1}^{2n}\dfrac{1}{\sigma^2}\eps^2_s-\dfrac{\left(\sum_{s=1}^{n}\eps_s \alpha^{s-1}\right)^2}{{\sigma^2}\sum_{s=1}^{n}\alpha^{2(s-1)}}-\dfrac{\left(\sum_{s=n+1}^{2n}\eps_s \alpha^{s-n-1}\right)^2}{{\sigma^2}\sum_{s=n+1}^{2n}\alpha^{2(s-n-1)}}.
\end{align*} 

Therefore, $\Loss^\ast(\bS)-\Loss(\bS; \tau^\ast)\sim\chi^2_2$. Moreover, note $\Loss(\bS,\varnothing)=\sum_{s=1}^{2n}(x_s-\hat\theta\alpha^{s-1})^2/\sigma^2$, where we omit the changepoint and as there is a single parameter $\theta$ to estimate, let $\eta=\alpha^n$, we have
\begin{align*}
\hat\theta=
\dfrac{\theta_1}{1+\eta^2}+\dfrac{\theta_2\eta}{1+\eta^2}+\dfrac{\sum_{s=1}^{2n}\eps_s\alpha^{s-1}}{\sum_{s=1}^{2n}\alpha^{2(s--1)}},
\end{align*}
Thus, by simply algebra calculation, we obtain
$
\Loss(\bS,\varnothing)-\Loss(\bS, \tau^\ast)\sim \chi^2_{1}(v)
$
, where the non-centrality parameter 
\begin{align}
\nu=\dfrac{(\theta_2-\eta\theta_1)^2(1-\eta^2)}{(1+\eta^2)(1-\alpha^2)}=\dfrac{\Delta^2(1+\alpha^{2n})}{(1-\alpha^{2n})(1-\alpha^2)}\ge\dfrac{\Delta^2}{(1-\alpha^{2n})(1-\alpha^2)}
\end{align}

Follow the same argument in the proof of Lemma \ref{lemma.drop1.A} we obtain that,  as long as $5\le z\le \nu/4$,
\begin{align*}
&\p\left(\Loss\left({\bS};\varnothing\right)-\Loss^\ast\left({\bS}\right)\le z + \Loss^\ast\left({\bS}\right)-\Loss\left({\bS, \tau^\ast}\right)\right)\\
\le& 
\p\left(\Loss\left({\bS};\varnothing\right)-\Loss\left({\bS};\tau^\ast\right)\le 2z\right)+\p\left(\Loss^\ast\left({\bS}\right)-\Loss\left({\bS, \tau^\ast}\right)\ge z\right)
\\
\le&\p\left(\chi^2_1(\nu)\le 2z\right)+\p\left(\chi^2_2\ge z\right)\\
\le&\exp\left(-\dfrac{(1+\nu-2z)^2}{4+8\nu}\right)+\exp\left(-\dfrac{z-2\sqrt{z}}{2}\right)\\
\le&2\exp\left(-\dfrac{z}{20}\right),
\end{align*}
\end{proof}

 Note that  Lemmas \ref{lemma.addk.C}, \ref{lemma.addk2.C} and \ref{lemma.drop1.C} hold for any $\bS\in \{\bS_{1,n}(t)\}$, $\{\bS_{2,n}(t)\}$ and $\{\bS_{\Delta,n}(t)\}$, respectively. Therefore, it is straightforward to obtain Propositions \ref{prop:spike1} and \ref{prop:spike2}.

\subsection{Proof of Theorem \ref{thm:spike}}

Similar to the proof of Theorem \ref{thm:mean}, we take $\beta=(2+\epsilon)\log T$ and 
$$n_j=\min\left\{\dfrac{1}{2}\log_{\alpha}\left(1-\dfrac{\Delta_j^2}{4(1-\alpha^2)(\beta+a(\beta,T))}\right), \delta_j. \delta_{j+1}\right\},$$
 where $a(\beta, T)=(\beta-2\log T)/4$ as indicates in Proposition \ref{prop:slope1} and let $b(\beta, T)=(\beta-2\log T)/4(2m^\ast+1)$. 

Therefore, since we assume 
$\dfrac{\Delta_T^2}{(1-\alpha^{2\delta^T})(1-\alpha^2)}\ge (8+5\epsilon)\log T$, we will have each $n_j$ achieves the minumum value at
\begin{align}
\dfrac{1}{2}\log_{\alpha}\left(1-\dfrac{\Delta_j^2}{4(1-\alpha^2)(\beta+a(\beta,T))}\right)=\dfrac{1}{2}\log_{\alpha}\left(1-\dfrac{\Delta_j^2}{(1-\alpha^2)(8+5\epsilon)\log T}\right).
\end{align}

 Combined with the assumption that $$\log_{\alpha}\left(1-\dfrac{\Delta_T^2}{4(1-\alpha^2)(\beta+a(\beta,T))}\right)\le {T^{2/(8m^\ast+6+\epsilon)}},$$ which leads to $2\max_j n_j\le T^{1/4}$. Therefore, it is straightforward to verify that as long as $T$ is large enough, we have $\beta=(2+\epsilon)\log T\ge \max\{\gamma^{(1)}_{T}, \max_{j}\gamma^{(2)}_{n_j}\}$,
where
\begin{equation*}
\gamma_T^{(1)}=\max\left\{(2+\epsilon)\log T, 2\log n+8\sqrt{16+2\log T}+32, 2\log T + 32(2m^\ast + 1)\right\},
\end{equation*} 

\begin{equation*}
 \gamma_{n_j}^{(2)}=\max\left\{(8m^\ast+6+\epsilon)\log (2n_j), 2\log (2n_j)+64(2m^\ast+1)\right\},
 \end{equation*}
as defined in Proposition \ref{prop:spike1}.

Next,by Proposition \ref{prop:spike1}, we can work out $p_1(\gamma, n), p_2(\gamma, n), p_3(\gamma, n)$ and $p_4(\gamma, n)$.
Since $\beta=(2+\epsilon)\log T$,  it is straightforward that 
 $$\min\bigg\{p_1(\beta, T)/2, p_2(\beta, T), p_3\left(\beta, \max_j n_j\right), p_4\left(\beta, \max_j n_j\right)\bigg\} \le T^{-\epsilon/16(2m^\ast+1)}.$$

Moreover, note that $\bar \mS=\min_j\dfrac{\Delta_j^2}{(1-\alpha^{2n_j})(1-\alpha^2)}\ge 4\left(\beta+a(\beta, T)\right)\ge 20$ as $T$ is large enough, we have 
 $$p_5\left(\bar\mS, \beta+a\left(\beta,T\right)\right)\le 2T^{-(2+5\epsilon/4)/20} \le 2T^{-\epsilon/(32m^\ast+16)}.$$

  Therefore, 
 \begin{align*}
&\p\bigg(\hat{m}=m, \min_j\dfrac{\Delta_j^2}{\left(1-\alpha^2\right)\left(1-\alpha^{2|\htau_j-\tau^\ast_j|}\right)}\ge (8+5\epsilon)\log T \bigg) \\
 \ge&1-(m^\ast+1)p_1\left(\beta,T\right)-(m^\ast+1) p_2(\beta, T)-m^\ast p_3\left(\beta, \max_j n_j\right)
 \\&-m^\ast p_4\left(\beta, \max_j n_j\right)\\&-m^\ast p_5\left(\sigst,\beta+a(\beta,T)\right)\\
 \ge& 1-(7m^\ast+3)T^{-\epsilon/(32m^\ast+16)}.
 \end{align*}
\qed

\medskip
\subsection{Proof of Corollary \ref{coro:spike}}
Let $c_5 \ge  8+5\epsilon$, and $0<c_6^{-1}\le (1-D)c_5$, since $\alpha^{T^{2/(8m^\ast+6+\epsilon)}}\le D < 1$, we have (\ref{spike.require.1}) and (\ref{spike.require.2}) hold. Applying Theorem \ref{thm:spike} with $c_5\ge 8+5\epsilon$, since $m^\ast=o(\log T)$,  as $T\ra \infty$, we obtain
 \begin{align*}
 &\p\bigg(\hat{m}=m, \min_j\dfrac{\Delta_j^2}{\left(1-\alpha^2\right)\left(1-\alpha^{2|\htau_j-\tau^\ast_j|}\right)}\ge c_5\log T \bigg) \\
 \ge& 1-(7m^\ast+3)T^{-\epsilon/(32m^\ast+16)}\ra 1.
 \end{align*}
\qed

\section{Orthogonal basis techniques for change-in-slope}\label{subsec:basis}

This appendix provides additional technical lemmas needed for the change-in-slope problem. We believe the orthogonal basis representation and maxima inequality of correlated Gaussian variables to be of independent interest and therefore present them in this separate appendix.

Without loss of generality, we re-index the $2n$ points in a local segment $\bS=x_{t+1:t+2n}$ as $\boldsymbol{x}=(x_1,\dots, x_{2n})^T$ with a single true changepoint at $\tau^\ast=n$. Let $\boldsymbol{f}=(f_1,\dots,f_{2n})^T$ denotes the vector of the linear signals with a change of slope at $\tau^\ast$, e.g

$$
f_i=\begin{cases}
\theta_0+ \dfrac{\theta_1-\theta_0}{n}i,  &i=1,\dots,n;\\[20pt]
\theta_1+ \dfrac{\theta_2-\theta_1}{n}(i-n), \quad &i=n+1,\dots,2n,
\end{cases}
$$
where $\theta_0$, $\theta_1$ and $\theta_2$ are unknown parameters, and $\bve=(\eps_1,\dots,\eps_{2n})^T$ denote the vector of Gaussian stochastic noises. Therefore $\boldsymbol{x}=\boldsymbol{f}+\bve$.  The following basis representation in the $2n$-dimensional vector space will be used to approximate $\boldsymbol{x}$. 

\subsection{Orthogonal basis.}

By algebra calculation, we can sequentially calculate the following basis representation for the $2n-$vector $\boldsymbol{x}$.

\medskip 

{\sc Basis Representation:}
\begin{itemize}
\item[1.] Constant basis representation: $\psi_{(C)}=\left(\psi_{(C)}(1),\dots,\psi_{(C)}(2n)\right)^T$ with $\psi_{(C)}(i)=(2n)^{-1/2}$. 
\item[2.] Linear basis: $\psi_{(L)}=\left(\psi_{(L)}(1),\dots,\psi_{(L)}(2  n)\right)^T$, with
$$\psi_{(L)}(i)=\sqrt{\frac{12}{2n(2n-1)(2n+1)}} \left(i-\dfrac{2n+1}{2}\right).$$ 
Note that $\psi_{(L)}$ is orthonormal to $\psi_{(C)}$.

\item[3.] Basis corresponding to $\tau^\ast=n$: $\psi_{(\tau^\ast)}=\left(\psi_{(\tau^\ast)}(1),\dots,\psi_{(\tau^\ast)}(2n)\right)^T$, with
$$
\psi_{(\tau^\ast)}(i)=\begin{cases}
-\sqrt{\dfrac{3(n+1)}{n(4n^2-1)(2n^2+1)(n-1)}}\bigg[(4n-1)i-n(2n+1)\bigg], \,  i=1,\dots,n;\\[20pt]
\sqrt{\dfrac{3(n-1)}{n(4n^2-1)(2n^2+1)(n+1)}}\bigg[(4n+1)i-3n(2n+1)\bigg], \, i=n+1,\dots,2n.
\end{cases}
$$
 Note that $\psi_{(\tau^\ast)}$ is orthonormal to both $\psi_{(C)}$ and $\psi_{(L)}$.

\item[4.] Basis  $\psi_{(\tau_1)}$ corresponding to adding an additional change $\tau_{1}$ on $\bS$, where $2\le \tau_1\le 2n$ and $\tau_1\neq \tau^\ast$. For example, if $2\le \tau_1 \le n-1$, then

\begin{align*}
\psi_{(\tau_1)}(i)=\begin{cases}
-A_n\sqrt{\dfrac{n-\tau_1}{n(n-1)(2n^2+1)\tau_1(\tau_1-1)}} 
\bigg(a_n i-b_n\bigg), \, i=1,\dots,\tau_1;\\[20pt]
 A_n\sqrt{\dfrac{\tau_1(\tau_1-1)}{n(2n^2+1)(n-\tau_1)(n-1)}}\bigg(c_ni
 -d_n\bigg), \, i=\tau_1+1,\dots,n;\\[20pt]
 -A_n\sqrt{\dfrac{\tau_1(\tau_1-1)(n-\tau_1)(n-1)}{n(2n^2+1)}}\bigg[3i-(5n+1)\bigg], \, i=n+1,\dots,2n.
\end{cases}
\end{align*}

where 
\begin{align*}
 \hspace*{-1cm}  A_n=&\sqrt{3}(8n^3\tau_1-4n^3-13n^2\tau_1^2+9n^2\tau_1+4n^2+5n\tau_1^3\\
&-6n\tau_1^2+5n\tau_1-2n+\tau_1^3-5\tau_1^2+2\tau_1+2)^{-1/2},\\
a_n=&(4n^3+4n^2\tau_1-4n^2-5n\tau_1^2+5n\tau_1+2n-\tau_1^2+3\tau_1-2),\\
b_n=&\tau_1(4n^3-3n^2\tau_1+3n^2-2n\tau_1+4n-\tau_1+1),\\
c_n=&9n^2-5n\tau_1+n-\tau_1+2,\\
d_n=&7n^3-3n^2\tau_1+2n^2-2n\tau_1+3n-\tau_1.
\end{align*}
Similarly we can write $\psi_{(\tau_1)}$ for $n+1\le \tau_1\le 2n$. Note that $\psi_{(\tau_1)}$ is orthonormal to $\psi_{(C)}$, $\psi_{(L)}$ and $\psi_{(\tau^\ast)}$. 

\item[5.] Basis  $\psi_{(\tau_j)}$, $j=2,3,\dots$, corresponding to adding a $j$-th  change $\tau_{j}$ on $\bS$ after $\tau_1,\dots, \tau_{j-1}$,   where  $2\le \tau_{j} \le 2n$ and $\tau_j\neq \tau_{1:(j-1)} \mbox{ or } n$. Moreover $\psi_{(\tau_j)}$ is orthonormal to $\psi_{(C)}$, $\psi_{(L)}$, $\psi_{(\tau^\ast)}$ and $\psi_{(\tau_{1:(j-1)})}$.

\end{itemize}

The formulas for $\psi_{(C)}, \psi_{(L)}$ and $\psi_{(\tau^\ast)}$ were also given in \cite{baranowski2016narrow}. We derive $\psi_{(\tau_1)}$ as it will be used in the proof of Lemmas \ref{lemma.max.Gaussian.A} and \ref{lemma.max.Gaussian.B}. The formulas for $\psi_{(\tau_j)}$  can be calculated by  applying the Gram-Schmidt procedure to make the vector $\nu_{(\tau_j)}$ (linear with a kink at $\tau_j$) orthogonal to $\psi_{(C)}, \psi_{(L)}$, $\psi_{(\tau^\ast)}$and $\psi_{(\tau_{1})},..., \psi_{(\tau_{j-1})}$, where
$$\nu_{(\tau_j)}(i)=\begin{cases}0,\quad i=1,2,\dots, \tau_j\\
i-\tau_j \quad i=\tau_j+1,\dots,2n,\end{cases}$$ 

We define $\bS_{\tau_1}= \{1,\dots, 2n\}\setminus \{1, n\}$;  $\bS_{\tau_2}= \{1,\dots, 2n\}\setminus \{1, \tau_1, n\}$ for any given $\tau_1$; and $$\bS_{\tau_{k+1}}= \{1,\dots, 2n\}\setminus \{1, \tau_1,\tau_2,\dots,\tau_k, n\}, \quad \mbox{given } \tau_1, \dots, \tau_k;$$ which are the sets of possible locations for $\tau_1$, $\tau_2$ and $\tau_{k+1}$ on $\bS$, respectively. To distinguish each of $\psi_{(\tau_j)}$, we write $\psi_{(i,j)}$ as the basis formulas for $\psi_{(\tau_j)}$ at locations $i$, where $j\in \{1,2,\dots,k\}$ and $i\in \bS_{\tau_j}$.

For each orthogonal basis $\psi_{(\cdot)}$, define the coefficients that correspond to $\boldsymbol{x}$ projected onto it as $x_{{(\cdot)}}=\left<\boldsymbol{x}, \psi_{(\cdot)}\right>=\left<\boldsymbol{f},\psi_{(\cdot)}\right>+\left<\bve,\psi_{(\cdot)}\right>=f_{(\cdot)}+\eps_{(\cdot)}$. We have the following straightforward properties for the signal components $f_{(\cdot)}$ and the noise components $\eps_{(\cdot)}$.
\begin{itemize}
\item[(1)] $f_{(\tau_j)}=0$  for $j=1,2,\dots,k$. 
\item[(2)] Each of $\eps_{(\cdot)}$ are i.i.d with distribution $\cN(0,\sigma^2)$. Without loss of generality, we assume $\sigma=1$ for the rest of the section. 

\item[(3)] For any two possible locations $i$ and $j$ in $\bS_{\tau_k}$,  we have $\e \{\eps_{(i,k)}, \eps_{(j,k)}\}=\mbox{corr} \{\eps_{(i,k)}, \eps_{(j,k)}\}=\left<\psi_{(i,k)},\psi_{(j,k)}\right>$.

\end{itemize}

The cost function of fitting changes within $\bS$ therefore can be expressed using above basis representation, for example:
\begin{align*}
\Loss^\ast(\bS)
=&\|\boldsymbol{x}-f_{(C)}\psi_{(C)}-f_{(L)}\psi_{(L)}-f_{(\tau^\ast)}\psi_{(\tau^\ast)}\|^2,\\
 \Loss\left(\bS, \varnothing\right) =&\|\boldsymbol{x}-x_{(C)}\psi_{(C)}-x_{(L)}\psi_{(L)}\|^2,\\
\Loss\left(\bS, \tau^\ast\right)
=&\|\boldsymbol{x}-x_{(C)}\psi_{(C)}-x_{(L)}\psi_{(L)}-x_{(\tau^\ast)}\psi_{(\tau^\ast)}\|^2,\\
\Loss(\bS; \tau_{1:k}, \tau^\ast)=&\bigg\|\boldsymbol{x}-x_{(C)}\psi_{(C)}-x_{(L)}\psi_{(L)}-x_{(\tau^\ast)}\psi_{(\tau^\ast)}- \sum_{j=1}^k x_{(\tau_j)}\psi_{(\tau_j)}\bigg\|^2.
\end{align*}

\begin{lemma}\label{lem.chi-square.slope} 
The following differences in cost follow  chi-square distributions:

$$\Loss^\ast(\bS)-\Loss\left(\bS;\tau^\ast\right)\sim\chi^2_{3}$$, $$\Loss^\ast({\bS})-\Loss({\bS};\tau_{1:k},\tau^\ast)\sim\chi^2_{k+3}$$ and $$\Loss(\bS,\varnothing)-\Loss(\bS,\tau^\ast)\sim\chi^2_{1}(\nu),$$
where $$\nu=\dfrac{(\theta_2-2\theta_1+\theta_0)^2}{n^2}\dfrac{n(n+1)(n-1)(2n^2+1)}{12(2n-1)(2n+1)}.$$
\end{lemma}
\begin{proof}
Applying the above properties and basis representation of loss functions, it is straightforward that 
\begin{align*}
\Loss^\ast(\bS)-\Loss\left(\bS;\tau^\ast\right)=&\|\boldsymbol{x}-f_{(C)}\psi_{(C)}-f_{(L)}\psi_{(L)}-f_{(\tau^\ast)}\psi_{(\tau^\ast)}\|^2\\
&-\bigg\|\boldsymbol{x}-x_{(C)}\psi_{(C)}-x_{(L)}\psi_{(L)}-x_{(\tau^\ast)}\psi_{(\tau^\ast)}\bigg\|^2\\
=&\eps^2_{(C)}+\eps^2_{(L)}+\eps^2_{(\tau^\ast)}\sim \chi^2_{3},
\end{align*}
and
\begin{align*}
\Loss^\ast({\bS})-\Loss({\bS};\tau_{1:k},\tau^\ast)=&\|\boldsymbol{x}-f_{(C)}\psi_{(C)}-f_{(L)}\psi_{(L)}-f_{(\tau^\ast)}\psi_{(\tau^\ast)}\|^2\\
&-\bigg\|\boldsymbol{x}-x_{(C)}\psi_{(C)}-x_{(L)}\psi_{(L)}-x_{(\tau^\ast)}\psi_{(\tau^\ast)}- \sum_{j=1}^k x_{(\tau_j)}\psi_{(\tau_j)}\bigg\|^2\\
=&\eps^2_{(C)}+\eps^2_{(L)}+\eps^2_{(\tau^\ast)}+\sum_{j=1}^{k}\eps^2_{(\tau_j)}\sim \chi^2_{k+3}.
\end{align*}
In addition,
\begin{align*}
\Loss(\bS,\varnothing)-\Loss(\bS,\tau^\ast)=&\|\boldsymbol{x}-f_{(C)}\psi_{(C)}-f_{(L)}\psi_{(L)}-f_{(\tau^\ast)}\psi_{(\tau^\ast)}\|^2\\
&-\|\boldsymbol{x}-x_{(C)}\psi_{(C)}-x_{(L)}\psi_{(L)}\|^2\\
=&x^2_{(\tau^\ast)}=\big\{f_{(\tau^\ast)}+\eps_{(\tau^\ast)}\big\}^2\sim \chi^2_1\left(f^2_{(\tau^\ast)}\right),
\end{align*}
where $f^2_{\tau^\ast}=\left<\boldsymbol{f},\psi_{(\tau^\ast)}\right>=(2n^2+1)\dfrac{(\theta_2-2\theta_1+\theta_0)^2}{n^2}\dfrac{n(n+1)(n-1)}{12(2n-1)(2n+1)}.$
\end{proof}

Moreover, one should noticed that if we consider $\bS'$ as a local region consisting no true changepoint, by similar arguments, we have the following lemma.
\begin{lemma} \label{lem.chi-square2.slope}
$\Loss^\ast(\bS')-\Loss(\bS'; \tau_{1:k})\sim\chi^2_{k+2}$ and
$\Loss^\ast\left(\bS'\right)-\Loss({\bS'};\varnothing)\sim \chi^2_{2}$.
\end{lemma}

\subsection{Maxima of correlated Gaussian variables.}  In this step, we prove a lemma that provides the upper bound of probability tail  for the maxima of a series of Gaussian random variables, i.e., $\max_{\tau_1\in\bS_{\tau_1}}\{\eps_{(\tau_1)}\}$, $\max_{\tau_2\in\bS_{\tau_2}}\{\eps_{(\tau_2)}\}$ for given $\tau_1$, and $\max_{\tau_3\in\bS_{\tau_3}}\{\eps_{(\tau_3)}\}$ for given $\tau_1$ and $\tau_2$. 

We first introduce the following Lemma \ref{lemma.davies}, which is a direct adaptation from a result in \cite{davies1977hypothesis}.
\begin{lemma}\label{lemma.davies} Let $\cG(t)$ be a Gaussian process indexed by $t\in [a,b]$, with expectation $0$ and covariance function $\e[\cG(t_1)\cG(t_2)]=\rho(t_1, t_2)$. Let $$\rho_{11}(t_1)=\dfrac{\partial^2\rho(t_1,t_2)}{\partial t_2^2}\bigg\rvert_{t_2=t_1}.$$  Then for any $z>0$:
\begin{equation}
\p\left(\sup_{ t} \cG(t)>z\right)\le \Phi(-z)+\dfrac{1}{2\pi}\exp\left(-\dfrac{z^2}{2}\right)\int_{a}^{b} |\rho_{11}(t)|^{1/2}\dt t,
\end{equation}
where $\Phi(\cdot)$ denotes the cumulative distribution function of $\cN(0,1)$.
\end{lemma}

Based on Lemma \ref{lemma.davies}, we can prove the following useful lemmas.

\begin{lemma}\label{lemma.max.Gaussian.A} There exists positive constants $C_1$, such that for any $z>0$,
\begin{equation}\label{ieq.gausain0.slope}
\p\left(\max_{\tau_1\in\bS_{\tau_1}} \eps_{(\tau_1)}>z\right)< C_1\exp\left(-\dfrac{z^2}{2}\right)\log (2n),
\end{equation}
\end{lemma}

\begin{proof}
Note that the collection of random variables $\{\eps_{(i,1)}\}$ for $i\in \bS_{\tau_1}$, which are $\bve$ projecting onto all the possible locations of $\tau_1$,  are jointly Gaussian with covariance $\e \{\eps_{(i,1)}, \eps_{(j,1)}\}=\left<\psi_{(i,1)},\psi_{(j,1)}\right>$, as each of them is a linear combination of i.i.d Gaussian variables $\eps_{1},\dots, \eps_{2n}$.

  Let $\mR_{\tau_1}=[2,n-1]\cup[n+1,  2n]$. Define a function $\rho(x,y)$ on $\mR_{\tau_1}\times \mR_{\tau_1}$ with continuous second derivatives  with respect to both components, such that for any pair $(i, j)\in \bS_{\tau_1}\times \bS_{\tau_1}$, we have $\rho(i,j)=\mbox{corr}(\eps_{(i,1)},\eps_{(j,1)})$. For example, we  could let $\rho(x,y)$ be the function that replace the discrete pair of variables $(i,j)$ in the formula of $\left<\psi_{(i,1)},\psi_{(j,1)}\right>$ by continuous pair of variables $(x,y)$. In this way,  by algebra calculation, we have
\begin{align}\label{eq.rho.slope2}
\rho(x,y)=
\begin{cases}
\sqrt{\dfrac{x(x-1)(n-y)}{y(y-1)(n-x)}}\dfrac{D_{n,1}(x,y)}{\sqrt{D_{n,1}(x) D_{n,1}(y)}},\hfill 2\le x\le y\le n-1;\\[20pt]
\sqrt{\dfrac{(n-y)(x-\tau_1)}{(n-x)(y-\tau_1)}}\dfrac{E_{n,1}(x,y)}{\sqrt{E_{n,1}(x) E_{n,1}(y)}},\hfill   x \le n-1 \mbox{ and } y\ge n+1;\\[20pt]
\sqrt{\dfrac{(2n-y)(2n-y+1)(x-n)}{(2n-x)(2n-x+1)(x-n)}}\dfrac{F_{n,1}(x,y)}{\sqrt{F_{n,1}(x) F_{n,1}(y)}}, \hfill  n+1\le x\le y\le 2n,
\end{cases}
\end{align}
 where we write

\noindent$D_{n,1}(x,y)= 2x - 2n - 3xy - 2xn + 7yn + xy^2 + 4xn^2 - 4xn^3 + 5yn^2 - y^2n + 12yn^3 - 2y^2 + 4n^2 - 4n^3 - 9y^2n^2 - 4xyn^2 + 5xy^2n - 5xyn + 2;
 $  
 $$D_{n,1}(x)=D_{n,1}(x,x), \quad D_{n,1}(y)=D_{n,1}(y,y);$$ 

 \noindent$E_{n,1}(x,y)=2x + 2y - 4n - xy - 7xn + 9yn + 4xn^2 + 4xn^3 + 4yn^2 + 36yn^3 + 7n^2 - 8n^3 - 12n^4 - 20xyn^2 + 2$, 
$$E_{n,1}(x)=E_{n,1}(x,x), \quad E_{n,1}(y)=E_{n,1}(y,y);$$ 

\noindent$F_{n,1}(x,y)=2y - 2n - 3xy + 7xn + 8yn + x^2y - xn^2 - x^2n - 8xn^3 - 2yn^2 - 8yn^3 - 2x^2 - 6n^2 + 2n^3 + 4n^4 + x^2n^2 + 16xyn^2 - 5x^2yn + xyn + 2$;
$$F_{n,1}(x)=F_{n,1}(x,x), \quad  F_{n,1}(y)=F_{n,1}(y,y).$$

If $x>y$, note that $\rho(x,y)=\rho(y,x)$, which completes the defination of $\rho(x,y)$. Therefore, we can construct a Gaussian process $\cG(t)$ indexed by  $\mR_{\tau_1}$, with mean $0$ and covariance function $\e[\cG(t_1) \cG(t_2)]=\rho(t_1, t_2)$, such that $\cG(i)=\eps_{(i)}$ for $i\in \bS_{\tau_1}$.

Notice that for any $2\le x\le n-1$, we obtain
\begin{align*}
\rho_{11}(x) = &\dfrac{1}{4}\left(\dfrac{1}{n-x}+\dfrac{1}{x}+\dfrac{1}{x-1}\right)^2-\dfrac{1}{2}\bigg\{\dfrac{1}{(n-x)^2}-\dfrac{1}{x^2}-\dfrac{1}{(x-1)^2}\bigg\}\\
&+\dfrac{1}{2 D_{n,1}(x)}\left(\dfrac{1}{n-x}+\dfrac{1}{x}+\dfrac{1}{x-1}\right)
\bigg\{\dfrac{\partial D_{n,1}(x)}{\partial x}-2\dfrac{\partial D_{n,1}(x,y)}{\partial y}\bigg\vert_{y=x}\bigg\}\\
&+\dfrac{1}{2D_{n,1}(x)}\bigg\{2\dfrac{\partial^2 D_{n,1}(x,y)}{\partial y^2}\bigg\vert_{y=x}- \dfrac{\partial^2 D_{n,1}(x)}{\partial x^2}\bigg\}\\
&-\dfrac{1}{2 D^2_{n,1}(x)}\left(\dfrac{1}{n-x}+\dfrac{1}{x}+\dfrac{1}{x-1}\right)
\bigg\{2\dfrac{\partial D_{n,1}(x)}{\partial x}\dfrac{\partial D_{n,1}(x,y)}{\partial y}\bigg\vert_{y=x}-\dfrac{3}{2}\left(\dfrac{\partial D_{n,1}(x)}{\partial x}\right)^2\bigg\}.
\end{align*}
 Let $g(x)=\dfrac{2}{x-1}+\dfrac{2}{n-x}$, applying the Maple program \pkg{Psdgcd} \citep{han2016proving} which can prove polynomial inequalities using symbolic computation, we prove that:
\begin{equation}\label{ieq.Dn.slope1}
\bigg\vert\dfrac{\partial D_{n,1}(x,y)/\partial y\vert_{y=x}}{D_{n,1}(x)}\bigg\vert \le g(x), \quad \bigg\vert\dfrac{\partial D_{n,1}(x)/\partial x}{D_{n,1}(x)}\bigg\vert \le g(x),
\end{equation}
and 
\begin{equation}\label{ieq.Dn.slope2}
  \bigg\vert\dfrac{\partial^2 D_{n,1}(x,y)/\partial y^2\vert_{y=x}}{D_{n,1}(x)}\bigg\vert \le g(x), \quad \bigg\vert\dfrac{\partial^2 D_{n,1}(x)/\partial x^2}{D_{n,1}(x)}\bigg\vert \le g(x).
\end{equation}

Hence, for $2\le x\le n-1$, the following inequality holds for some  $c>0$,
$$|\rho_{11}(x)| < c \max\left\{\dfrac{1}{(x-1)^2}, \dfrac{1}{(n-x)^2}\right\}.$$   
In a similar way, we can also obtain that for $\tau_1+1\le x\le n-1$, the following inequality holds for some $c>0$, 
$$|\rho_{11}(x)| < c \max\left\{\dfrac{1}{(x-n)^2}, \dfrac{1}{(2n-x)^2}\right\}.$$

Altogether, we have that there exists an absolute constant $C_1$ that does not depend on $n,\tau_1$, such that for any $x\in \mR_{\tau_1}$ the following inequality holds: 
$$|\rho_{11}(x)|\le 
4\pi^2 C_1^2 \max\left\{\dfrac{1}{(x-1)^2},  \dfrac{1}{(x-n)^2}, \dfrac{1}{(2n-x)^2}
\right\}.
$$

By Lemma \ref{lemma.davies}, we have
\begin{align*}
\p\left(\sup_{ t\in  \mR_{\tau_1}} \cG(t)>z\right)=& \dfrac{1}{2\pi}\exp\left(-\dfrac{z^2}{2}\right)\int_{\mR_{\tau_1}} |\rho_{11}(t)|^{1/2}\dt t\\
=&  C_1\exp\left(-\dfrac{z^2}{2}\right)\log (2n-4).
\end{align*}

As a result, we obtain
\begin{align*}
\p\left(\max_{\tau_2\in \bS_{\tau_1}} \eps_{(\tau_2)}>z\right)\le\p\left(\sup_{t\in \mR_{\tau_1}} \cG(t)>z\right)= C_1\exp\left(-\dfrac{z^2}{2}\right)\log (2n-4),
\end{align*}
which proves (\ref{ieq.gausain0.slope}).

\end{proof}

\begin{lemma}\label{lemma.max.Gaussian.B} There exists absolute constants $C_2, C_3$, such that for any given $\tau_1$,
\begin{equation}\label{ieq.gausain1.slope}
\p\left(\max_{\tau_2\in \bS_{\tau_2}} \eps_{(\tau_2)}>z\right)< C_2\exp\left(-\dfrac{z^2}{2}\right)\log (2n),
\end{equation}
and for any given $\tau_1$, $\tau_2$,
\begin{equation}\label{ieq.gausain2.slope}
\p\left(\max_{\tau_3\in \bS_{\tau_3}} \eps_{(\tau_3)}>z\right)< C_3\exp\left(-\dfrac{z^2}{2}\right)\log (2n).
\end{equation}

\end{lemma}

\begin{proof}
 For a given $\tau_1$, note that  $\{\eps_{(i,2)}\}_{i\in \bS_{\tau_2}}$  are jointly Gaussian with covariance $\e \{\eps_{(i,2)}, \eps_{(j,2)}\}=\left<\psi_{(i,2)},\psi_{(j,2)}\right>$. Due to the symmetry of the local region $\boldsymbol x$ (having a change in the middle),  we only need to deal with $2\le \tau_1 \le n-1$.  

  Let $\mR_{\tau_2}=[2, \tau_1-1]\cup[\tau_1+1,n-1]\cup[n+1,  2n]$. Define $\rho(x,y): \mR_{\tau_2}\times \mR_{\tau_2}\ra [-1,1]$ as the function that replace the discrete pair $(i,j)$ in the formula of $\left<\psi_{(i,2)},\psi_{(j,2)}\right>$ by continuous pair $(x,y)$. In this way,  by algebra calculation, we have
\begin{align}\label{eq.rho.slope1}
\rho(x,y)=
\begin{cases}
\sqrt{\dfrac{x(x-1)(\tau_1-y)}{y(y-1)(\tau_1-x)}}\dfrac{D_{n,2}(x,y)}{\sqrt{D_{n,2}(x) D_{n,2}(y)}},\hfill 2\le x\le y\le \tau_1-1;\\[20pt]
\sqrt{\dfrac{(n-y)(x-\tau_1)}{(n-x)(y-\tau_1)}}\dfrac{E_{n,2}(x,y)}{\sqrt{E_{n,2}(x) E_{n,2}(y)}},\hfill \tau_1+1 \le x\le y\le n-1;\\[20pt]
\sqrt{\dfrac{(2n-y)(2n-y+1)(x-n)}{(2n-x)(2n-x+1)(x-n)}}\dfrac{F_{n,2}(x,y)}{\sqrt{F_{n,2}(x) F_{n,2}(y)}}, \hfill \quad n+1\le x\le y\le 2n.
\end{cases}
\end{align}
where we write

\noindent$
D_{n,2}(x,y)=2n - 11 n^2 \tau_1^2  + 11 n^2\tau_1^3  - 22 n^3 \tau_1^2  - 5n^2\tau_1^4  + 13n^3\tau_1^3  - 8n^4\tau_1^2  + 2n^2y^2  - 4n^3y^2  + 4n^4y^2
- 3\tau_1^2y^2  + 6\tau_1^3  y^2  - 3\tau_1^4y^2  + 2 n x - 7 n\tau_1^2  + 7 n^2  \tau_1 + 6 n\tau_1^3  + 5 n^3\tau_1 - n\tau_1^4  + 12 n^4\tau_1
   - 2 n^2x + 4n^3x - 4n^4x - 2ny^2  + 3\tau_1^2 y - 6\tau_1^3y + 3\tau_1^4y - 2 n^2 + 4n^3  - 4n^4  - 7n\tau_1^2x
   + 7n^2\tau_1 x + 6n\tau_1^3 x + 5 n^3\tau_1 x - n\tau_1^4x + 12n^4\tau_1 x + 15 n\tau_1^2 y - 4 n^2\tau_1 y - 28 n\tau_1^3 y
   + 8 n^3\tau_1 y + 9 n \tau_1^4 y - 8 n^4\tau_1 y - 2 nxy^2  + 3\tau_1 x y^2  + 6 \tau_1^2  x y - 3\tau_1^3  x y - 11 n^2\tau_1^2 x
   + 11 n^2\tau_1^3 x - 22 n^3\tau_1^2 x - 5n^2\tau_1^4 x + 13n^3\tau_1^3 x - 8n^4\tau_1^2 x - 4 n\tau_1^2 y^2  - 7 n^2\tau_1 y^2
   + 29 n^2\tau_1^2 y + 12 n\tau_1^3 y^2  - 32 n^2\tau_1^3 y - 5n^3\tau_1 y^2  + 31n^3 \tau_1^2 y - 6 n\tau_1^4 y^2  + 15 n^2\tau_1^4 y
   - 39n^3\tau_1^3 y - 12n^4 \tau_1 y^2  + 24 n^4\tau_1^2 y + 2n^2 xy^2  - 4n^3xy^2  + 4n^4xy^2  - 6\tau_1^2xy^2  + 3\tau_1^3xy^2
   + 4 n\tau_1 y - 3\tau_1 x y - 8n^2\tau_1^2y^2  + 9n^2\tau_1^3y^2  + 9n^3\tau_1^2 y^2  - 11n\tau_1 x y - 14n^2 \tau_1^2xy^2
   + 9 n \tau_1 x y^2  + 18 n\tau_1^2  x y - 19 n^2\tau_1 x y - 7 n\tau_1^3xy - 13n^3\tau_1 xy - 8n^4\tau_1 x y
   - 13 n\tau_1^2  x y^2  + 14 n^2 \tau_1 xy^2  + 20 n^2\tau_1^2 xy + 6 n\tau_1^3xy^2  - 5 n^2\tau_1^3 xy + 4n^3\tau_1 xy^2
+ 13 n^3\tau_1^2  x y
 ,$  
  $$D_{n,2}(x)=D_{n,2}(x,x), D_{n,2}(y)=D_{n,2}(y,y);$$ 
\noindent $E_{n,2}(x,y)=10x^2y^2n^2\tau_1 - 5x^2y^2n^2 - 6x^2y^2n\tau_1^2 + 8x^2y^2n\tau_1 - x^2y^2n - 3x^2y^2\tau_1^2 + 3x^2y^2\tau_1 - 8x^2yn^3\tau_1 + 4x^2yn^3 - 6x^2yn^2\tau_1^2 - 4x^2yn^2\tau_1 + 5x^2yn^2 + 6x^2yn\tau_1^3 - 3x^2yn\tau_1^2 - 9x^2yn\tau_1 + 3x^2yn + 3x^2y\tau_1^3 - 3x^2y\tau_1 - 8x^2n^4\tau_1 + 4x^2n^4 + 21x^2n^3\tau_1^2 - 13x^2n^3\tau_1 - 4x^2n^3 - 9x^2n^2\tau_1^3 + 15x^2n^2\tau_1^2 - 10x^2n^2\tau_1 + 2x^2n^2 - 6x^2n\tau_1^3 + 15x^2n\tau_1^2 - 5x^2n\tau_1 - 2x^2n - 3x^2\tau_1^3 + 3x^2\tau_1^2 - 18xy^2n^3\tau_1 + 9xy^2n^3 + 10xy^2n^2\tau_1^2 - 17xy^2n^2\tau_1 + 6xy^2n^2 + 8xy^2n\tau_1^2 - 13xy^2n\tau_1 + 3xy^2n + 3xy^2\tau_1^2 - 3xy^2\tau_1 + 24xyn^4\tau_1 - 12xyn^4 - 8xyn^3\tau_1^2 + 22xyn^3\tau_1 - 9xyn^3 - 4xyn^2\tau_1^2 + 23xyn^2\tau_1 - 12xyn^2 - 6xyn\tau_1^3 - 3xyn\tau_1^2 + 12xyn\tau_1 - 3xyn - 3xy\tau_1^3 + 3xy\tau_1 - 8xn^4\tau_1^2 + 4xn^4\tau_1 - 13xn^3\tau_1^2 + 17xn^3\tau_1 + 9xn^2\tau_1^3 - 19xn^2\tau_1^2 + 8xn^2\tau_1 + 6xn\tau_1^3 - 11xn\tau_1^2 + 7xn\tau_1 + 3x\tau_1^3 - 3x\tau_1^2 + 9y^2n^3\tau_1^2 - 9y^2n^3 - 5y^2n^2\tau_1^3 + y^2n^2\tau_1^2 + 5y^2n^2\tau_1 - y^2n^2 - y^2n\tau_1^3 + 2y^2n\tau_1^2 + y^2n\tau_1 - 2y^2n - 12yn^4\tau_1^2 + 12yn^4 + 4yn^3\tau_1^3 - 5yn^3\tau_1^2 - 4yn^3\tau_1 + 5yn^3 + 5yn^2\tau_1^3 - 7yn^2\tau_1^2 - 5yn^2\tau_1 + 7yn^2 + 3yn\tau_1^3 - 3yn\tau_1 + 4n^4\tau_1^3 + 4n^4\tau_1^2 - 4n^4\tau_1 - 4n^4 - 4n^3\tau_1^3 - 4n^3\tau_1^2 + 4n^3\tau_1 + 4n^3 + 2n^2\tau_1^3 + 2n^2\tau_1^2 - 2n^2\tau_1 - 2n^2 - 2n\tau_1^3 - 2n\tau_1^2 + 2n\tau_1 + 2n;
 $
 $$E_{n,2}(x)=E_{n,2}(x,x), \quad E_{n,2}(y)=E_{n,2}(y,y);$$
 \noindent$F_{n,2}(x,y)=2y - 4n + 2\tau_1 + 17n^2\tau_1^2 - 17n^2\tau_1^3 + 53n^3\tau_1^2 - 32n^3\tau_1^3 + 32n^4\tau_1^2 + 16n^4\tau_1^3 - 52n^5\tau_1^2 + 28n^5\tau_1^3 - 44n^6\tau_1^2 - 3xy + 7xn + 6yn + 2y\tau_1 + 3n\tau_1 + x^2y - 8xn^2 + x^2n + 7xn^3 - 8xn^4 - 14xn^5 + 16xn^6 - 6yn^2 + 6yn^3 - 12yn^4 - 12yn^5 + 16yn^6 - 2x^2\tau_1 - 5y\tau_1^2 + y\tau_1^3 - n\tau_1^2 - 2n^2\tau_1 + 4n\tau_1^3 - 23n^3\tau_1 - 44n^4\tau_1 + 4n^5\tau_1 + 44n^6\tau_1 + 16n^7\tau_1 - 2x^2 - 6n^4 + 12n^5 + 4n^6 - 8n^7 - 5\tau_1^2 + \tau_1^3 - 2x^2n^2 + x^2n^3 + 4x^2n^4 - 2x^2n^5 + 5x^2\tau_1^2 - x^2\tau_1^3 + 9xyn^2 - 6x^2yn - 8xyn^3 + 30xyn^4 - 32xyn^5 + 12xy\tau_1^2 + x^2y\tau_1 - 6xy\tau_1^3 - 22xn\tau_1^2 + 21xn^2\tau_1 - 6x^2n\tau_1 + 8xn\tau_1^3 + 39xn^3\tau_1 - xn^4\tau_1 - 58xn^5\tau_1 - 32xn^6\tau_1 - 35yn\tau_1^2 + 36yn^2\tau_1 + 18yn\tau_1^3 + 49yn^3\tau_1 - 6yn^4\tau_1 - 62yn^5\tau_1 - 32yn^6\tau_1 + 7x^2yn^2 - 12x^2yn^3 + 10x^2yn^4 - 7x^2y\tau_1^2 + 5x^2y\tau_1^3 - 41xn^2\tau_1^2 + 13x^2n\tau_1^2 - 15x^2n^2\tau_1 + 35xn^2\tau_1^3 - 36xn^3\tau_1^2 - 10x^2n\tau_1^3 - 10x^2n^3\tau_1 - 11xn^3\tau_1^3 + 53xn^4\tau_1^2 + 5x^2n^4\tau_1 - 38xn^4\tau_1^3 + 70xn^5\tau_1^2 + 4x^2n^5\tau_1 - 59yn^2\tau_1^2 + 37yn^2\tau_1^3 - 35yn^3\tau_1^2 - 18yn^3\tau_1^3 + 64yn^4\tau_1^2 - 38yn^4\tau_1^3 + 70yn^5\tau_1^2 + 4xyn - 3xy\tau_1 + 7xn\tau_1 + 13yn\tau_1 + 18x^2n^2\tau_1^2 - 2x^2n^2\tau_1^3 - x^2n^3\tau_1^2 + 7x^2n^3\tau_1^3 - 11x^2n^4\tau_1^2 - 11xyn\tau_1 + 13x^2yn^2\tau_1^2 - 17x^2yn^2\tau_1^3 + 37x^2yn^3\tau_1^2 + 29xyn\tau_1^2 - 13xyn^2\tau_1 + 2x^2yn\tau_1 - 25xyn\tau_1^3 + 37xyn^3\tau_1 + 94xyn^4\tau_1 + 64xyn^5\tau_1 - 10xyn^2\tau_1^2 + 5x^2yn\tau_1^2 - 8x^2yn^2\tau_1 + 15xyn^2\tau_1^3 - 77xyn^3\tau_1^2 - 23x^2yn^3\tau_1 + 58xyn^3\tau_1^3 - 122xyn^4\tau_1^2 - 20x^2yn^4\tau_1 + 2$, 
 $$F_{n,2}(x)=F_{n,2}(x,x), \quad F_{n,2}(y)=F_{n,2}(y,y).$$

The formula of $p(x,y)$ for other cases, such as $x<\tau_1<y$ and $x<n<y$ can be derived similarly. Moreover, if $x>y$, note that $\rho(x,y)=\rho(y,x)$. Therefore, we can define a Gaussian process $\cG(t)$ indexed by  $\mR_{\tau_2}$, with mean $0$ and covariance function $\e[\cG(t_1) \cG(t_2)]=\rho(t_1, t_2)$, such that $\cG(i)=\eps_{(i,2)}$ for $i\in \bS_{\tau_2}$.

Notice that for any $2\le x\le \tau_1-1$, we obtain
\begin{align*}
\rho_{11}(x) = &\dfrac{1}{4}\left(\dfrac{1}{\tau_1-x}+\dfrac{1}{x}+\dfrac{1}{x-1}\right)^2-\dfrac{1}{2}\bigg\{\dfrac{1}{(\tau_1-x)^2}-\dfrac{1}{x^2}-\dfrac{1}{(x-1)^2}\bigg\}\\
&+\dfrac{1}{2 D_{n,2}(x)}\left(\dfrac{1}{\tau_1-x}+\dfrac{1}{x}+\dfrac{1}{x-1}\right)
\bigg\{\dfrac{\partial D_{n,2}(x)}{\partial x}-2\dfrac{\partial D_{n,2}(x,y)}{\partial y}\bigg\vert_{y=x}\bigg\}\\
&+\dfrac{1}{2D_{n,2}(x)}\bigg\{2\dfrac{\partial^2 D_{n,2}(x,y)}{\partial y^2}\bigg\vert_{y=x}- \dfrac{\partial^2 D_{n,2}(x)}{\partial x^2}\bigg\}\\
&-\dfrac{1}{2 D^2_n(x)}\left(\dfrac{1}{\tau_1-x}+\dfrac{1}{x}+\dfrac{1}{x-1}\right)
\bigg\{2\dfrac{\partial D_{n,2}(x)}{\partial x}\dfrac{\partial D_{n,2}(x,y)}{\partial y}\bigg\vert_{y=x}-\dfrac{3}{2}\left(\dfrac{\partial D_{n,2}(x)}{\partial x}\right)^2\bigg\}.
\end{align*}
 Let $g(x)=\dfrac{2}{x-1}+\dfrac{2}{\tau_1-x}$, with the help of  \pkg{Psdgcd}, we can prove that:
\begin{equation}\label{ieq.Dn.slope2.1}
\bigg\vert\dfrac{\partial D_{n,2}(x,y)/\partial y\vert_{y=x}}{D_{n,2}(x)}\bigg\vert \le g(x), \quad \bigg\vert\dfrac{\partial D_{n,2}(x)/\partial x}{D_{n,2}(x)}\bigg\vert \le g(x),
\end{equation}
and 
\begin{equation}\label{ieq.Dn.slope2.2}
  \bigg\vert\dfrac{\partial^2 D_{n,2}(x,y)/\partial y^2\vert_{y=x}}{D_{n,2}(x)}\bigg\vert \le g(x), \quad \bigg\vert\dfrac{\partial^2 D_{n,2}(x)/\partial x^2}{D_{n,2}(x)}\bigg\vert \le g(x).
\end{equation}

Hence, for $2\le x\le \tau_1-1$, the following inequality holds,
$$|\rho_{11}(x)| < c \max\left\{\dfrac{1}{(x-1)^2}, \dfrac{1}{(\tau_1-x)^2}\right\}$$ 
where $c>0$ is a constant.  In a similar way, we can also obtain that for $\tau_1+1\le x\le n-1$ 
$$|\rho_{11}(x)| < c \max\left\{\dfrac{1}{(n-x)^2}, \dfrac{1}{(x-\tau_1)^2}\right\}$$ 
 and for $n+1\le x<y\le 2n$
$$|\rho_{11}(x)| < c \max\left\{\dfrac{1}{(x-n)^2}, \dfrac{1}{(2n-x)^2}\right\},$$ 
 where  $c$ is some universal positive constants.

Altogether, we have that there exists an absolute constant $C_2$ that does not depend on $n,\tau_1$, such that for any $x\in \mR_{\tau_2}$ the following inequality holds: 
$$|\rho_{11}(x)|\le 
4\pi^2 C_2^2 \max\left\{\dfrac{1}{(x-1)^2}, \dfrac{1}{(\tau_1-x)^2}, \dfrac{1}{(x-n)^2}, \dfrac{1}{(2n-x)^2}
\right\}.
$$

By Lemma \ref{lemma.davies}, we have
\begin{align*}
\p\left(\sup_{ t\in  \mR_{\tau_2}} \cG(t)>z\right)=& \dfrac{1}{2\pi}\exp\left(-\dfrac{z^2}{2}\right)\int_{\mR_{\tau_2}} |\rho_{11}(t)|^{1/2}\dt t\\
=&  C_2\exp\left(-\dfrac{z^2}{2}\right)\log (2n-6).
\end{align*}

As a result, we obtain
\begin{align*}
\p\left(\max_{\tau_2\in \bS_{\tau_2}} \eps_{(\tau_2)}>z\right)\le\p\left(\sup_{t\in \mR_{\tau_2}} \cG(t)>z\right)= C_2\exp\left(-\dfrac{z^2}{2}\right)\log (2n-6),
\end{align*}
which proves (\ref{ieq.gausain1.slope}).

The proof of (\ref{ieq.gausain2.slope}) can be obtained similarly, thus is omitted here.
\end{proof}

\subsection{Tight probabilistic bounds  for fitting too many changes.} \label{sec:D.3}
In this step, we provide the following two lemmas, which is much tighter than simply applying Bonferroni correction for the probability of maximum over the events that we fit too many changes.

\begin{lemma}\lbl{lemma.max1.B} As long as $\gamma>\max\bigg\{2\log(2n)+32\log(C_1'''\log (2n)), 2\log(2n)+972(2m^\ast+1)\bigg\}$, where $C_1'''$ is a positive constant only related to $C_1$ in (\ref{ieq.gausain0.slope}),
we have 
$$
\p\left(\max_{\tau_1}\left\{\Loss^\ast(\bS)-\Loss_n(\bS; \tau_{1}, \tau^\ast)\right\}\ge \dfrac{\gamma-2\log (2n)}{6}\right)\le \exp\left(-\dfrac{\gamma-2\log(2n)}{24(2m^\ast+1)}\right)
$$
\end{lemma}

\begin{proof}
By the orthonormality of the basis and the properties (i), 
\begin{align*}
\Loss^\ast(\bS)-\Loss(\bS; \tau_1)\le& \Loss^\ast(\bS)-\Loss(\bS; \tau_{1},\tau^\ast)\\
=&\eps^2_{(C)}+\eps^2_{(L)}+\eps^2_{(\tau^\ast)}+\eps^2_{(\tau_1)}
\end{align*}

This implies
\begin{align*}
\p\left(\Loss^\ast(\bS)-\min_{\tau_1\in \bS_{\tau_1}}\Loss(\bS; \tau_1)\ge z\right)\le& \p\left(\eps^2_{(C)}+\eps^2_{(L)}+\eps^2_{(\tau^\ast)}+\max_{\tau_1\in\bS_{\tau_1}}\left(\eps^2_{(\tau_1)}\right)\ge z\right)\\
\le & \p\left(\eps^2_{(C)}+\eps^2_{(L)}+\eps^2_{(\tau^\ast)}+\eps^2_{(\tau_1)}+\sup_{t\in \mR_{\tau_1}} \cG^2(t)\ge z\right),
\end{align*}
where $\cG(t)$ is the  continuous Gaussian process constructed based on $\eps_{(\tau_1)}, \tau_1\in\bS_{\tau_1}$ as in the proof of Lemma \ref{lemma.max.Gaussian.A}.

Let $Z_1=\eps^2_{(C)}+\eps^2_{(L)}+\eps^2_{(\tau^\ast)}$ and $Z_2=\sup_{t\in \mR_{\tau_1}} \cG^2(t)$. Note that $\eps_{(C)},\eps_{(L)}$ and $\eps_{(\tau^\ast)}$  are i.i.d random variables with $\cN(0,1)$ distribution, and are all independent to $\eps_{(\tau_1)}$ for any $\tau_1\in \bS_{\tau_1}$. Therefore $Z_1\sim \chi^2_3$ and is independent to $Z_2$.

Using the arguments from the proof of Lemma 1 in \cite{laurent2000adaptive}, we can upper bound the logarithm of Laplace transform of $Z_1$: 
$$
\log\{\e \exp\left(u(Z_1-3)\right)\}\le \frac{3u^2}{1-2u}, \quad \mbox{ for } 0<u<1/2.
$$

Note that for all $z$
$$\p\left(\sup_{t\in \mR_{\tau_1}} \cG(t)\ge z\right) \le \p\left(\sup_{t\in \mR_{\tau_1}} |\cG(t)|\ge z\right) \le 2\p\left(\sup_{t\in \mR_{\tau_1}} \cG(t)\ge z\right). $$

Therefore we have the probability density function $f(z)$ for $\sup_{t\in \mR_{\tau_1}} |\cG(t)|$ is upper bounded by $C_1'\log (2n-6)x\exp(-z^2/2)$ when $z$ is large enough, where $C_1'$ is a positive constant only depending on $C_1$. This further leads to an upper bound on the Laplace transform of $Z_2$:
$$\log\{\e\exp(uZ_2)\}\le \log(C_1''\log(2n-6))-\dfrac{\log(1-2u)}{2}.$$

For $0<u<1/2$, we have 
\begin{align*}
\log\e\exp\left[u\{Z_2-1-\log(C_1''\log(2n-6))\}\right]\le&-u-\dfrac{\log(1-2u)}{2}\\
\le&\dfrac{u^2}{1-2u}
\end{align*}
As a result,  let $Z= Z_1+Z_2-4-\log(C_1''\log(2n-6))$, then
\begin{align*}
\log\{\e \exp\left(u(Z)\right)\}
\le \dfrac{5u^2}{1-2u}
\end{align*}

By Lemma 8.2 in \cite{birge2001alternative}, if 
$$
\log\left(\e e^{uZ}\right)\le \dfrac{v^2u^2}{1-bu}, \quad \mbox{for }  0<t< b^{-1}
$$
then for any positive $x$,
$$
\p\left(Z\ge bx+2v\sqrt{x}\right)\le \exp(-x).
$$
Hence, we have for any given $\tau_1$, as long as $z\ge 162$,
\begin{align*}
\p\left(\Loss^\ast(\bS)-\min_{\tau_1\in \bS_{\tau_1}}\Loss(\bS; \tau_1, \tau_{2})\ge z\right)\le & \p\left(Z\ge z-4-\log(C_1''\log(2n-4))\right)\\
\le& \p\left(Z\ge z-(C_1''+4)-\log\log(2n)\right)\\
\le& C_1'''\log(2n)\exp\left(-\dfrac{4z}{9}\right),
\end{align*}
where $C_1'''$ is a positive constant only depend on $C_1$.


Taking $z=\dfrac{\gamma-2\log (2n)}{6(2m^\ast+1)}$, then as long as 
$$\gamma>\max\bigg\{2\log(2n)+32\log(C_1'''\log (2n)), 2\log(2n)+972(2m^\ast+1)\bigg\}$$ we have  
$$C_1''' \log(2n) \exp\left(-\dfrac{4z}{9}\right)\le \exp\left(-\dfrac{\gamma-2\log(2n)}{24(2m^\ast+1)}\right), $$
which completes the proof.

\end{proof}

\begin{lemma}\lbl{lemma.max2.B} As long as $\gamma>\max\left\{240, 24\log(2C_2'''\log (2n))\right\}$, where $C_2'''$ is a positive constant only related to $C_2$ in (\ref{ieq.gausain1.slope}),
we have 
$$
\p\left(\max_{\tau_1,\tau_2}\left\{\Loss^\ast(\bS)-\Loss_n(\bS; \tau_{1:2}, \tau^\ast)\right\}\ge \gamma-\dfrac{\gamma-2\log (2n)}{6}\right)\le \exp\left(-\dfrac{\gamma-3\log(2n)}{3}\right)
$$
\end{lemma}

\begin{proof}

Using a similar argument to that of the proof of Lemma \ref{lemma.max1.B}, we have that for any given $\tau_1$  and $\tau_2$, as long as $z\ge 200$,
 \begin{align*}
\p\left(\Loss^\ast(\bS)-\min_{\tau_2\in\bS_{\tau_1}}\Loss(\bS; \tau_1, \tau_{2})\ge z\right)
\le C_2'''\log(2n)\exp\left(-\dfrac{9z}{20}\right),
\end{align*}
where $C_2'''$ is a positive constant only depends on $C_2$.

Consider all the $2n-2$ possible locations for the first change $\tau_1$, by Bonferroni correction, 
\begin{align*}
\p\left(\Loss^\ast(\bS)-\min_{\tau_1, \tau_2}\Loss(\bS; \tau_{1}, \tau_{2})\ge z\right)\le& (2n-2)\p\left(\Loss^\ast(\bS)-\min_{\tau_2\in \bS_{\tau_1}}\Loss(\bS; \tau_1, \tau_{2})\ge z\right)\\
\le& C_2''' 2n \log(2n) \exp\left(-\dfrac{9z}{20}\right).
\end{align*}

Taking $z=\gamma-\dfrac{\gamma-2\log (2n)}{6}=\dfrac{5\gamma+2\log(2n)}{6}$, then as long as $$\gamma>24\log(2C_2'''\log (2n)),$$ we have  
$$C_2''' 2n \log(2n) \exp\left(-\dfrac{9z}{20}\right)\le \exp\left(-\dfrac{\gamma-3\log(2n)}{3}\right), $$
which completes the proof.
\end{proof}

\begin{lemma}\lbl{lemma.max3.B} As long as $\gamma>\max\{132, 24\log(2C_3'''\log (2n))\}$, where $C_3'''$ is a positive constant only related to $C_3$ in (\ref{ieq.gausain2.slope})
we have 
$$
\p\left(\Loss^\ast(\bS)-\min_{\tau_{1:3}}\Loss(\bS; \tau_{1:3}, \tau^\ast)\ge 2\gamma-\dfrac{\gamma-2\log(2n)}{6}\right)\le \exp\left(-\dfrac{\gamma-3\log(2n)}{3}\right).
$$
\end{lemma}
\begin{proof}
 Using the similar argument as in the proof of Lemma \ref{lemma.max2.B}, we have for any given $\tau_1$  and $\tau_2$, as long as $z\ge 242$,
 \begin{align*}
\p\left(\Loss^\ast(\bS)-\min_{\tau_3}\Loss(\bS; \tau_1, \tau_{2})\ge z\right)
\le C_3'''\log(2n)\exp\left(-\dfrac{10z}{22}\right),
\end{align*}
where $C_3'''$ is a positive constant only depends on $C_3$.

Consider all the $(2n-2)\times(2n-3)$ possible locations for the first two changes $\tau_1$ and $\tau_2$, by Bonferroni correction, 
\begin{align*}
\p\left(\Loss^\ast(\bS)-\min_{\tau_1, \tau_2,\tau_3}\Loss(\bS; \tau_{1}, \tau_{2}, \tau_3)\ge z\right)\le& (2n-2)\p\left(\Loss^\ast(\bS)-\min_{\tau_3}\Loss(\bS; \tau_1, \tau_{2}, \tau_3)\ge z\right)\\
\le& C_3''' (2n)^2 \log(2n) \exp\left(-\dfrac{10z}{22}\right).
\end{align*}

Taking $z=2\gamma-\dfrac{\gamma-2\log (2n)}{6}=\dfrac{11\gamma+2\log(2n)}{6}$, as long as $$\gamma>24\log(2C_3'''\log (2n)),$$ we have 
$$C_3''' (2n)^2 \log(2n) \exp\left(-\dfrac{10z}{22}\right)\le \exp\left(-2\dfrac{(\gamma-3\log(2n))}{3}\right),$$
which completes the proof.
\end{proof}


\bibliographystyle{Chicago}





\end{document}